\RequirePackage{fix-cm}
\makeatletter
\def\cl@chapter{}
\makeatother
\documentclass[smallextended, envcountsect]{svjour3}
\usepackage[a4paper,hmargin=3cm,vmargin=3cm]{geometry}

\input{preamble}

\smartqed
\journalname{}
\begin{document}

\maketitle

\begin{abstract}
This paper studies stochastic algorithms for minimizing paraconvex functions, a function class that generalizes weakly convex functions and includes, for instance, H\"older smooth functions and compositions of convex functions with H\"older smooth maps. We first establish the convergence of the stochastic subgradient method and the stochastic proximal subgradient method for unconstrained and composite paraconvex optimization, respectively. The analysis is then extended to a broad family of algorithms that access the objective through stochastic models satisfying certain approximation-quality and continuity assumptions. The main principle that underlies the convergence analysis is that the algorithms under consideration can be interpreted as perturbed descent methods on the generalized Moreau envelope of the original paraconvex problem.
\end{abstract}



%
\section{Introduction}

Stochastic optimization underlies many modern aspects of learning from data, where the goal is to learn a decision rule given a limited number of samples that still can generalize well to the entire data population. Solving such problems amounts to minimizing the \emph{regularized population risk}:

\begin{equation}\label{eq:intro_so}
    \min_{x \in \R^n} F(x) \coloneqq f(x) + \varphi(x), \quad \text{where}~ f(x) = \E[\xi \sim \pi]{f(x, \xi)}. \tag{SO}
\end{equation}

Here, $\xi$ encodes the data, which is assumed to follow some probability distribution $\pi$. The function $f$ evaluates the expected loss of the parametrized decision rule, while the regularizer $\varphi$ enforces constraints on the learned model $x$. In the setting where $f(., \xi)$ are smooth, the most successful and widely used algorithm to solve \cref{eq:intro_so} is the stochastic proximal {gradient} descent method, which can be traced back to the forward-backward splitting algorithm introduced by Lions and Mercier \cite{LionsMercier1979}:
\begin{align}\label{eq:psg_scheme}
\begin{cases}
    \text{Draw}~ \xi_t \sim \pi, \\
    \text{Update}~ x_{t+1} = \prox[\gamma_t \varphi]{x_t - \gamma_t \nabla f(x_t, \xi_t) },
\end{cases} \tag{PSG}
\end{align}
where $\gamma_t$ is the step size and $\prox[\gamma_t \varphi]{.}$ is the proximal mapping
\begin{align*}
    \prox[\gamma_t \varphi]{x} = \margmin[y]{ \varphi(y) + \frac{1}{2\gamma_t} \norm[2]{y-x} }.
\end{align*}

Nonsmooth problems can be solved by the same procedure, where the gradients are replaced by suitable generalized subdifferentials (see, for example, \cite[Chapter 8]{rockafellar_variational_1998}, \cite[Chapter 1]{mordukhovich_variational_2006}, or \cite[Chapter 2]{clarke_nonsmooth_1998}). Theoretical analysis of the scheme \cref{eq:psg_scheme} and related variants has been extensively studied in various settings, with convergence evaluated in terms of an appropriate accuracy or stationarity measure \cite{lan_optimal_2012,ghadimi_optimal_2012,ghadimi_optimal_2013,ghadimi_mini-batch_2016}.


Davis and Drusvyatskiy \cite{davis_stochastic_2019} establish for the first time the convergence of stochastic proximal subgradient methods, as well as other algorithms that have access to stochastic models of the function $f$ satisfying suitable approximation and Lipschitz properties, for a broad class of nonsmooth nonconvex functions, namely, \emph{weakly convex} functions \cite{vial_strong_1983}. The key ingredient for their analysis is the \emph{Moreau envelope} \cite{Moreau1965}, which is a smooth approximation of weakly convex problems. The gradient norm of the Moreau envelope then serves as a natural near-stationarity measure for monitoring the convergence of optimization algorithms. This work has since influenced many subsequent works on nonsmooth nonconvex optimization \cite{mai_convergence_2020,deng_minibatch_2021,alacaoglu_convergence_2021}.


The quadratic nature of weak convexity, however, is not always intrinsic to the geometry of nonconvex objectives. In this work, we focus on the class of \emph{paraconvex} functions \cite{Rolewicz1979GammaParaconvex}, which generalizes the class of weakly convex functions while still enjoying favorable variational properties that allow for the use of subgradient-based optimization methods. In particular, a function $f: \R^n \to \R \cup \set{+\infty}$ is said to be \emph{$(\rho,p)$-paraconvex}, for $\rho>0$ and $p>1$, if, for any $x,y\in\R^n$ and $t\in[0,1]$, the following inequality holds:
\begin{equation*}
    f(tx+(1-t)y)\le tf(x)+(1-t)f(y)+\frac{\rho}{p}t(1-t)\|x-y\|^p.
\end{equation*}

Formal definitions and properties of this function class are given in \cref{sec:paraconvexity}. Notably, this class includes \emph{H\"older smooth} functions, or more generally, functions given by the composition of a Lipschitz convex function with a H\"older smooth mapping; both are of great interest for the nonconvex nonsmooth optimization research community \cite{ghadimi_conditional_2019,ito_parameter-free_2023}.



Our analysis calls for a generalization of the Moreau envelope, from which a suitable near-stationarity measure can be built. Such a generalization is provided in \cref{sec:generalized_moreau}. Algorithms considered in this paper can be interpreted as perturbed descent methods on the generalized Moreau envelope.


\textbf{Contributions and organization of the paper.}
The goal of this paper is to study the convergence of several stochastic optimization methods for minimizing paraconvex objectives. In \cref{sec:preliminaries}, we formally define the function classes of interest and establish their main properties. We also introduce the generalized Moreau--Yoshida regularization and its associated generalized proximal mapping, whose generalized subdifferential allows us to define a natural near-stationarity measure that will be used throughout our convergence analysis. 

In \cref{sec:ssgd}, we first consider the unconstrained version of \cref{eq:intro_so}, corresponding to $\varphi=0$, and show that the stochastic subgradient method converges at the rate $\mathcal{O}(T^{1/p-1})$ for $(\rho,p)$-paraconvex objectives. In \cref{sec:proximal}, we establish the same convergence rate for the stochastic proximal subgradient method when $\varphi$ is a convex regularizer. A notable feature of our analysis is that we only require the stochastic subgradient to satisfy a \emph{bounded $p$-th moment} condition, which is weaker than the usual bounded second-moment assumption used in the analysis of weakly convex optimization. This suggests that paraconvex optimization may provide a suitable framework for real-world applications in the presence of heavy-tailed noise.

We then study a more general algorithmic framework, namely, stochastic model-based methods, in which the oracle provides access to the loss $f$ at a point $x$ through a stochastic model $f_x(\cdot,\xi)$, where $\xi\sim\pi$. This framework was also studied by Davis and Drusvyatskiy in the weakly convex setting \cite{davis_stochastic_2019}. It includes the stochastic proximal subgradient method as a special case, as well as the stochastic prox-linear method, a generalization of the Gauss--Newton method \cite{duchi_stochastic_2018}. We provide a comprehensive convergence analysis of this framework for minimizing paraconvex functions. Finally, we present numerical experiments demonstrating the benefits of using a paraconvex objective over a weakly convex one for robust phase retrieval with heavy-tailed noise.

\textbf{Related works.} We conclude the introduction with a discussion on the related literature. To the best of our knowledge, the only work that explicitly studies paraconvex optimization is the recent paper by Rahimi et al. \cite{rahimi_projected_2026}, which studies the projected subgradient method with several step-size schemes under the Hölderian error bound condition, which is related to the weak sharpness condition \cite{Davis2018SubgradientMF,burke_weak_1993}, or more generally to the classical \L{}ojasiewicz inequality \cite{bolte_lojasiewicz_2007}. Thanks to this error bound condition, local linear convergence to a global minimum can be established. Our work goes beyond these results by considering the general proximal subgradient method in the stochastic setting, without imposing the error bound assumption.

A recent line of work develops optimization schemes based on the high-order Moreau envelope and its associated high-order proximal mapping, which likewise fall within the framework of the generalized Moreau--Yosida regularization introduced in \cref{def:gen_moreau_2}. Their fundamental properties are studied in \cite{kabgani_fundamental_2026,kabgani_moreau_2025}. High-order proximal-point methods are studied in \cite{ahookhosh_minimizing_2026}, where the high-order Moreau envelope is used as a smoothing scheme. Ahookhosh and Nesterov \cite{ahookhosh_high-order_2024} study inexact high-order proximal-point methods for convex optimization, while inexact high-order smoothing schemes are studied by Kabgani and Ahookhosh in \cite{kabgani_itsdeal_2025,kabgani_itsopt_2026} for nonsmooth nonconvex optimization. An important contribution of these works is that, under the Kurdyka--\L{}ojasiewicz (KL) condition \cite{Attouch2008ProximalAM}, choosing a suitable regularization order depending on the KL exponent can yield linear convergence rates. The main challenge in this framework is that each iteration requires computing---if only approximately---a high-order proximal point and the corresponding high-order Moreau envelope. By contrast, in the algorithms considered in this paper, the generalized Moreau envelope and generalized proximal mapping enter only through the \emph{convergence analysis}, every iteration involves only the classical proximal mapping of a convex function.

We should also mention the line of studies on minimizing H\"older smooth functions. Inspired by the seminal work of Nesterov on universal gradient methods \cite{nesterov_gradient_2013} that can adapt to the smoothness level of convex objectives, Ghadimi-Lan-Zhang \cite{ghadimi_generalized_2019} introduce universal algorithms for the nonconvex smooth setting. Ghadimi \cite{ghadimi_conditional_2019} presents a unified analysis of the conditional gradient method for minimizing \cref{eq:intro_so} with $F$ H\"older smooth. To the best of our knowledge, this is the only work with complexity guarantees for stochastic composite H\"older smooth optimization. However, their stochastic algorithm \cite[Algorithm 4]{ghadimi_conditional_2019} requires a large batch size of order $\mathcal{O}(\epsilon^{-2})$ to reach an $\epsilon$-stationary point measured by the Frank-Wolfe gap.


\section{Preliminaries}\label{sec:preliminaries}
Throughout this paper, we work in the Euclidean space $\R^n$, equipped with its canonical scalar product $\langle \cdot,\cdot\rangle$ and the associated norm $\|\cdot\|$.
We denote by $\unitball$ the unit ball and by $\ball[x]{r}$ the open ball centered at $x$ with radius $r>0$. 

Given an extended-real-valued function $f:\R^n\to\R\cup\{+\infty\}$, we denote by $\dom f$ the set of points where $f$ is finite.
The \emph{Fr\'echet subdifferential} of $f$ at $\bar{x}\in\dom f$ is defined as
\begin{align*}
\widehat{\partial} f(x)
\coloneqq \Bigl\{\, v \in \R^n :
\liminf_{y \to x,\, y \neq x}
\frac{f(y) - f(x) - \langle v, y - x \rangle}{\|y - x\|} \ge 0
\Bigr\}.
\end{align*}

Following \cite{rockafellar_variational_1998}, this object is also referred to as the \emph{regular subdifferential}.
Elements of $\hat{\partial} f(\bar{x})$ are called \emph{Fr\'echet (regular) subgradients}.
If $f(\bar{x}) = +\infty$, we set $\hat{\partial}f(\bar{x})=\emptyset$.
Equivalently, $v$ is a Fr\'echet subgradient of $f$ at $\bar{x}$ if and only if
\begin{align*}
f(x) \ge f(\bar{x}) + \langle v, x - \bar{x}\rangle + o(\|x - \bar{x}\|),
\quad\text{with}\quad
\lim_{x \to \bar{x}} \frac{o(\|x - \bar{x}\|)}{\|x - \bar{x}\|} = 0.
\end{align*}

It is well known that the Fr\'echet subdifferential satisfies a \emph{fuzzy sum rule}; see \cite{Fabian1989} and \cite[Theorem~2.33]{mordukhovich_variational_2006}. In words, any Fr\'echet subgradient of $f+g$ at $\bar x$ can be approximated up to
arbitrary precision by the sum of a Fréchet subgradient of $f$ at a nearby point
and a Fréchet subgradient of $g$ at a nearby point.



The \emph{Mordukhovich limiting subdifferential} of $f$ at $\bar{x}$ is defined as
\begin{align*}
\partial f(x)
\coloneqq \left\{\, v \in \R^n \;:\;
\exists\, x_k \to x,\; f(x_k) \to f(x),\;
v_k \in \widehat{\partial} f(x_k),\; v_k \to v \,\right\}.
\end{align*}

Intuitively, \( \partial f(x) \) generalizes the notion of the gradient to nonsmooth settings. When \( f \) is continuously differentiable, the limiting subdifferential reduces to the singleton
\begin{align*}
\partial f(x) = \{\, \nabla f(x) \,\}.
\end{align*}

For convex functions, \( \partial f(x) \) coincides with the usual convex subdifferential of convex analysis.
The quantity
$
\mathrm{dist}(0, \partial f(x)) = \inf\{\, \|v\| : v \in \partial f(x) \,\}
$
thus measures how far the point \( x \) is from satisfying a first-order optimality condition. We say that $f$ is \emph{(Fréchet) regular} at $x$ if $\partial f(x) = \hat{\partial} f(x)$.

The following sum rule for the limiting subdifferential is classical (see \cite{rockafellar_variational_1998}):
if $f_1,f_2:\R^n\to\R\cup\{+\infty\}$ are lower semicontinuous (l.s.c.), with $f_2$ locally Lipschitz at $x\in\dom f_1\cap\dom f_2$, then
\begin{align*}
\partial (f_1+f_2)(x) \subseteq \partial f_1(x) + \partial f_2(x).
\end{align*}

\subsection{Paraconvex functions}\label{sec:paraconvexity}
In this section, we provide formal definitions for the classes of paraconvex functions studied in this paper, together with some of their key properties. The notions of paraconvexity of functions and multifunctions were introduced by S. Rolewicz in 1979 \cite{Rolewicz1979GammaParaconvex}, which have then been extensively studied by Rolewicz in a series of works \cite{Rolewicz1981Phi2Subdifferentiability,Rolewicz1999AlphaMonotone,Rolewicz2000AlphaParaconvex,Rolewicz2003PhiConvex,Rolewicz2005ParaconvexAnalysis}, by Jourani \cite{jourani_96,Jourani1996OpenMapping}, by Ngai and coauthors \cite{ngai_paraconvex_2008,van_ngai_variational_2022}, and find connections with some other works \cite{DaniilidisMalick2005LowerC1C2,Rockafellar1982Favorable}. The following definition is standard.

\begin{definition}\label{def:paraconvex}
    For given $p \in (1, 2]$, $\rho > 0$, a proper function $f: \R^n \to \R \cup \set{+\infty}$ is called $(\rho, p)$-paraconvex if for all $x, y \in \R^n$ and $t \in [0, 1]$, the following holds:
    \begin{equation}\label{ineq-paraconvex1}
        f(tx+(1-t)y)\le tf(x)+(1-t)f(y)+\frac{\rho}{p}t(1-t)\|x-y\|^p.
    \end{equation}
\end{definition}

The case $p=2$ is of particular interest, since it corresponds to the notion of \emph{weak convexity} \cite{vial_strong_1983}. The convergence of stochastic optimization algorithms for minimizing weakly convex functions has been established by Davis and Drusvyatskiy in their seminal work \cite{davis_stochastic_2019}, and has been extensively studied and developed since then. The case $p>2$ reduces to standard convexity, see, for example, \cite{Rolewicz1979GammaParaconvex,Rolewicz2000AlphaParaconvex}. Henceforth, we restrict our studies to the case $p \in (1, 2]$.


Paraconvexity also implies \emph{approximate convexity}, in the sense introduced by Ngai-Luc-Théra \cite{NgaiLucThera2000ApproxConvex}. Therefore, local Lipschitzness of paraconvex functions follows from the local Lipschiztness of approximate convex functions established in \cite[Proposition 3.2]{NgaiLucThera2000ApproxConvex}. Nice variational properties of paraconvex functions then follow, in the sense that all the usual subdifferentials coincide. The following proposition summarizes these properties.

\begin{proposition}
    Let $f: \R^n \to \R \cup \set{+\infty}$ be a proper l.s.c. $(\rho, p)$-paraconvex function with $p \in (1, 2]$ and $\rho > 0$. Then the following claims hold:
    \begin{itemize}
        \item[(i)] $f$ is locally Lipschitz over $\operatorname{int}\dom{f}$.
        \item[(ii)] The Fréchet subdifferential, the limiting subdifferential, and the Clark subdifferential of $f$ coincide. Throughout the rest of the paper, for paraconvex functions, we will simply refer to their subdifferential. As a consequence, the subdifferential of $f$ is non-empty over $\operatorname{int}\dom{f}$. 
    \end{itemize}
\end{proposition}

The coincidence of subdifferentials of approximate convex functions has been established in \cite[Theorem 3.6]{NgaiLucThera2000ApproxConvex}. The non-emptiness property then follows directly from the local Lipschitzness as a well-established result, see \cite[Proposition 2.1.2]{clarke_optimization_1990}. The following subdifferential characterization of paraconvex functions has been established in \cite{ngai_paraconvex_2008}.

\begin{theorem}[Theorem 8 in \cite{ngai_paraconvex_2008}]\label{charac-thm} Let $f:\R^n\rightarrow\R\cup\{+\infty\} $ be a proper l.s.c. function. For $p \in (1, 2], \rho>0$, consider the three statements:
\begin{itemize}
\item[(i)] The function $f$ is $(\rho, p)$-paraconvex.
\item[(ii)] For all $x,y\in \R^n,$ $u\in\partial f(x),$ one has
\begin{equation}\label{charac-ineq1}
\langle u,y-x\rangle \le f(y)-f(x)+\frac{\rho}{p}\|x-y\|^p.
\end{equation}
\item[(iii)] For all $x,y\in\R^n;$ all $u\in \partial f(x),$ $v\in \partial f(y),$ one has
\begin{equation}\label{charac-ineq2}
\langle u-v,x-y\rangle\ge -\frac{2\rho}{p}\|x-y\|^p.
\end{equation}
\end{itemize}
One has $(i)\Rightarrow (ii)\Rightarrow (iii).$ Moreover, if (ii) holds then $f$ is $(\rho 2^{2-p}, p)$-paraconvex; while (iii) holds then $f$ is $(2\rho, p)$-paraconvex.
\end{theorem}

The preceding theorem implies immediately that if the function $f$ is of class $C^1$ and its derivative $\nabla f$ is continuously $(\nu,L)$-Hölderian on a closed convex set $C\subseteq \mathbb{R}^n,$ for some $\nu \in (0, 1]$, $L>0$, i.e.,
\begin{equation}\label{holder-deriv}
\|\nabla f(x)-\nabla f(y)\|\le L\|x-y\|^\nu\;\;\forall x,y\in C,
\end{equation}
then $f$ is $(2L(1+\nu), 1+\nu)$-paraconvex on $C$. The subdifferential characterization above also implies the paraconvexity of the class of composite functions $f = h \circ c$, with $h$ convex, globally Lipschitz and $c$ a smooth map with Hölder-continuous Jacobian.

\begin{proposition}\label{propo:composite_paraconvexity}
    Let $h: \R^n \to R$ a convex and $L_h$-Lipschitz function with $L_h > 0$; $g: \R^m \to \R^n$ be a smooth mapping with $(\nu, L_c)$-Hölder continuous Jacobian where $\nu \in (0, 1]$ and $L_c > 0$. Then the composite function $f = h \circ c$ is $(\rho, 1+\nu)$-paraconvex with $\rho = \frac{2^{1-\nu} L_hL_c}{1+\nu}$.
\end{proposition}

\begin{proof}
    For any $x \in \R^m$, let $J(x) \in \R^{n \times m}$ be the Jacobian of $c$ at $x$. The Hölder continuity of $J$ \cite{yashtini_global_2016} implies:
    \begin{align*}
        \norm{c(x) - c(y) - J(x) (y-x) } \leq \frac{L_c}{1+\nu} \norm[1+\nu]{x-y} ~ \forall x, y \in \R^m.
    \end{align*}

    Let $g \in \partial f(x)$. By the chain rule, there exists $v \in \partial h(c(x))$ such that $g = J(x)^T v$. The convexity and the Lipschitzness of $h$ imply:
    \begin{align*}
        f(y) - f(x) - \langle g, y-x \rangle
        & = h(c(y)) - h(c(x)) - \langle J(x)^T v, y-x \rangle \\
        & \geq \langle v, c(y) - c(x) - J(x)^T (y-x) \rangle \\
        & \geq - \norm{v} \norm{ c(y)-c(x)-J(x)(y-x) } \\
        & \geq -\frac{L_h L_c}{1+\nu} \norm[1+\nu]{x-y}.
    \end{align*}

    The claim follows from \cref{charac-thm}~(ii).
\end{proof}

\begin{example}\label[example]{example:composite_paraconvex}
    The robust phase retrieval problem \cite{davis_stochastic_2019,duchi_stochastic_2018} seeks to recover the true signal from perturbed measurements by solving the following weakly-convex optimization problem:
    \begin{align*}
        \min_{x \in \R^n} \frac{1}{m} \sum_{i=1}^m \abs{ \langle a_i, x \rangle^2 - b_i },
    \end{align*}
    with Gaussian channels $a_i \in \R^n$ and noisy measurements $b_i \in \R$. In cases where the channels and the perturbed noises are heavy-tailed, one can consider the following alternative formulation:
    \begin{align*}
        \min_{x \in \R^n} \frac{1}{m} \sum_{i=1}^m \abs{ \abs{\langle a_i, x \rangle}^p - b_i }.
    \end{align*}

    A simple argument show that the function $x \mapsto \abs{ \langle a, x \rangle }^p$ is $(\nu, L)$-Hölder smooth over $\R^n$ with $\nu=p-1$ and $L = p 2^{2-p} \norm[p]{a} $. Thus, paraconvexity of the objective function above follows from \cref{propo:composite_paraconvexity}. Numerical experiments demonstrating the advantage of using the paraconvex objective in terms of recovering the true signal in this heavy-tail setting are given in \cref{sec:experiment}.
\end{example}

We also consider a more general class of functions called $(\rho, \eta, p)$-paraconvex. This class of functions represents, for example, functions that are the sum of paraconvex and weakly convex functions. The formal definition is given below.

\begin{definition}\label{def:paraweak}
    A proper function $f: \R^n \to \R \cup \set{+\infty}$ is called $(\rho, \eta, p)$-paraconvex for $\rho > 0, \eta > 0$ and $p \in (1, 2]$ if $x \mapsto f(x) + \frac{\eta}{2} \norm[2]{x}$ is $(\rho, p)$-paraconvex. Equivalently, for all $x, y \in \R^n$, $t \in [0, 1]$, the following holds:
    \begin{equation}\label{eq:para-weak-convexity}
        f ( tx + (1-t)y ) \leq t f(x) + (1-t) f(y) + \frac{\rho t(1-t)}{p} \norm[p]{x-y} + \frac{\eta t(1-t)}{2} \norm[2]{x-y}.
    \end{equation}
\end{definition}

It is obvious that the class of $(\rho, \eta, p)$-paraconvex functions preserves all the nice variational properties of $(\rho, p)$-paraconvex functions in \cref{propo:composite_paraconvexity}. A $(\rho, \eta, 2)$-paraconvex function is then $(\rho+\eta)$-weakly convex. The following subdifferential characterization is also straightforward from the subdifferential characterization of paraconvex functions in \cref{charac-thm} and weakly convex functions \cite{davis_stochastic_2019,vial_strong_1983}.

\begin{lemma}\label{lem:charac_paraweak}
    Let $f:\R^n\rightarrow\R\cup\{+\infty\} $ be a proper l.s.c. function. For $p>1,$ $\rho>0,$ consider the three statements:
    \begin{itemize}
    \item[(i)] The function $f$ is $(\rho, \eta, p)$-paraconvex.
    \item[(ii)] For all $x,y\in \R^n,$ $u\in\partial f(x),$ one has
    \begin{equation}\label{charac-paraweak-ineq1}
    \langle u,y-x\rangle \le f(y)-f(x)+\frac{\rho}{p}\|x-y\|^p + \frac{\eta}{2} \norm[2]{x-y} .
    \end{equation}
    \item[(iii)] For all $x,y\in\R^n;$ all $u\in\partial f(x),$ $v\in \partial f(y),$ one has
    \begin{equation}\label{charac-paraweak-ineq2}
    \langle u-v,x-y\rangle\ge -\frac{2\rho}{p}\|x-y\|^p - \eta \norm[2]{x-y} .
    \end{equation}
    \end{itemize}
    One has $(i)\Rightarrow (ii)\Rightarrow (iii).$ Moreover, if (ii) holds then $f$ is $(\rho2^{2-p}, \eta, p)$-paraconvex; while (iii) holds then $f$ is $(2\rho, \eta, p)$-paraconvex.
\end{lemma}

\subsection{Generalized Moreau--Yosida regularization and stationarity measure}\label{sec:generalized_moreau}

The main goal in nonconvex optimization is the search for stationary points; monitoring an algorithm's progress toward stationarity therefore requires an appropriate \emph{stationarity measure}. In the smooth setting, a natural choice is the magnitude of the gradients or the generalized gradient mapping \cite{ghadimi_stochastic_2013} along the iterates. In the more general nonsmooth, weakly convex setting, Davis and Drusvyatskiy \cite{davis_stochastic_2019} pioneer the use of the gradient norm of the Moreau envelope \cite{Moreau1965} as a \emph{near-stationarity measure}: a small value certifies that the current iterate lies close to an approximate stationary point. A key property underlying this approach is that, for a suitably chosen regularization parameter, the Moreau envelope of a weakly convex function is $C^1$-smooth. Such a property does not hold for paraconvex functions. In this section, we introduce the \emph{generalized Moreau-Yoshida regularization}, its associated generalized proximal mapping, together with some of their key variational properties needed to define a near-stationarity measure for paraconvex functions. These tools have been studied by J.-P. Penot in the late $1990$s \cite{Penot1996Favorable,penot98}, and subsequently developed by Bernard and Thibault \cite{BernardThibault2004ProxRegularBanach,BernardThibault2005ProxRegularHilbert,BernardThibault2005UniformProxRegular}, by Jourani-Thibault-Zagrodny \cite{JouraniThibaultZagrodny2014}, by Ngai and Penot \cite{van_ngai_subdifferentiation_2016}, to name a few.

\begin{definition}\label{def:gen_moreau_2}
    For a proper l.s.c. function $f:\R^n\to\R\cup\{+\infty\}$, with {$r>0$}, an increasing continuous function $h: \R_+ \to \R_+$ such that $h(0) = 0$ and $h(t) \to +\infty$ as $t \to +\infty$, define
    \begin{equation*}\label{MR_1}
    \moreaup[r][f]{x} \coloneqq \inf_{y \in \R^n} \left\{f(y)+\frac{1}{r} h (\norm{x-y}) \right\}
    \end{equation*}
    as the generalized Moreau--Yosida regularization of $f$ with respect to the regularizing kernel $h(\norm{x-y})$. The corresponding generalized proximal mapping is defined as follows:
    \begin{equation*}\label{Proj_1}
    \gproxp[r][f]{x} \coloneqq \left\{ y\in\R^n:\;f(y)+\frac{1}{r} h(\norm{x-y}) = M^{ r}_f(x) \right\}.
    \end{equation*}
\end{definition}


The following results are from \cite{van_ngai_subdifferentiation_2016}.

\begin{theorem}\label{thm:generalized_moreau_paraweak}
    Let $f: \R^n \rightarrow \R^n \cup \{+\infty\}$ be a proper l.s.c. function; $h: \R_+ \to \R_+$ be an increasing continuous and locally Lipschitz function such that $h(0)=0$ and $h(t) \to +\infty$ as $t \to +\infty$. We have the following statements:
    \begin{itemize}
    \item[(i)] Assume that $f$ satisfies the growth condition:
    \begin{equation}\label{eq:growth_cond_1}
        f(x) \geq b - a h( \norm{x} ),~ \text{for some} ~ a > 0 ~ \text{and} ~ b \in \R, ~ \forall x \in \R^n \tag{G},
    \end{equation}
    and $h$ satisfies the following condition:
    \begin{equation}\label{eq:cond_h}
        \exists ~ c\ge 1 ~ \text{s.t. for any}~ d > 0, \exists ~ m > 0 ~ \text{s.t.} ~
        h(t+d) \leq ch(t)+m ~ \forall ~ t \in \R_+.\tag{H}
    \end{equation}
    Then for $0<r<(ac)^{-1},$ the function $M^{r}_f$ is locally Lipschitz, and  $P^{r}_f(x)\neq \emptyset$ for all $x\in\R^n$. Moreover, if $P^{r}_f(x)$ is a singleton then $M^{r}_f(x)$ is Fréchet differentiable at $x$.
    \item[(ii)] For $x\in\R^n,$ if $h$ satisfies the condition \cref{eq:cond_h} and is differentiable with $h'(0)=0$, $h'(t)>0$ for $t>0$, and $ P^{r}_f(x) $ is nonempty, then for all $z \in P^{r}_f(x) $,
    \begin{equation}\label{eq:subdif-regu-1}
    \hat\partial M^{r}_f(x) \subseteq \hat\partial f(z) \cap \set{ r^{-1} \nabla j_h (x-z)},
    \end{equation}
    where $j_h \coloneqq h \circ ||.||$. As a result, if $f$ satisfies the growth condition \cref{eq:growth_cond_1}, then for all $0<r<(ac)^{-1},$ either $\hat\partial M^{r}_f(x) $ is empty or it is singleton. Furthermore, when $\hat\partial M^{r}_f(x) $ is a singleton, the function $M^{r}_f$ is Fréchet differentiable at $x$.
    
    \item[(iii)] Assume $f$ satisfies the growth condition \cref{eq:growth_cond_1}. For $0<r<(ac)^{-1}$, and $h: \R_+ \to \R_+$ satisfies the assumptions in statement (ii), for all $x\in \R^n,$ the limiting subdifferential $\partial M^{r}_f (x) $ is nonempty, and there exists $z\in P^{r}_f (x) $ such that
    \begin{align}\label{eq:stationary_measure_1}
    d(0,\partial f(z)) \le d \left(0,\partial M^{r}_f(x) \right) &\coloneqq\inf\{\|v\|:\;\; v\in\partial M^{r}_f (x) \} \\
    &=r^{-1} \norm{ \nabla j_h (x-z) } \label{eq:limit-subdif-estim-1} \notag .
    \end{align}
    \end{itemize}
\end{theorem}

For a $(\rho, p)$-paraconvex function $f$, with $h(t) = t^p/p$, it is easy to see that the growth condition \cref{eq:growth_cond_1} is satisfied with $a = \rho$. Similarly, a $(\rho, \eta, p)$-paraconvex function satisfies \cref{eq:growth_cond_1} with $h(t) = t^p/p + t^2/2$ and $a=\rho+\eta$. Hence, \cref{thm:generalized_moreau_paraweak}~(iii) suggests that, with a sufficiently small regularization parameter and a suitable kernel $h$, the distance from $0$ to the limiting subdifferential of the generalized Moreau envelope can be interpreted as a near-stationarity measure for paraconvex functions. Therefore, convergence guarantees in subsequent sections will be presented in terms of this measure.

\section{Stochastic subgradient methods for unconstrained paraconvex optimization}\label{sec:ssgd}
We will begin by considering the following unconstrained optimization problem:
\begin{equation}\label{eq:para-prob}
\min_{x\in \R^n} F(x),
\end{equation}
where $F:\R^n\to \R$ is a real-valued l.s.c. and $(\rho, p)-$paraconvex function with $p \in (1, 2]$ and $\rho > 0$. We assume that the only access to $F$ is through a stochastic gradient oracle satisfying the assumption below:

\begin{assumption}\label[assumption]{assump:stochastic_grad_oracle}
Fix a probability space $\left( \Omega, \mathcal{F}, \pi \right)$, and equip $\R^n$ with the Borel $\sigma$-algebra. We make the following assumptions:
\begin{itemize}
\item[(i)] We can draw i.i.d. realizations $\xi_1, \xi_2, \dots$ from $\pi$.
\item[(ii)] There exists a measurable mapping $g:\R^n \times \Omega\rightarrow\R^n$ such that\\
$\mathbb{E}_{\xi\sim \pi} g(x,\xi)\in\partial F(x),$ for all $x\in\R^n$.
\item[(iii)] There exists a constant $L>0$ such that $\mathbb{E}_{\xi\sim P} \|g(x,\xi)\|^p\le L^p,$ for all $x\in\R^n$.
\end{itemize}
\end{assumption}

\Cref{assump:stochastic_grad_oracle}~(i) is standard in the literature of stochastic optimization, while \cref{assump:stochastic_grad_oracle}~(ii) and \cref{assump:stochastic_grad_oracle}~(iii) are weaker than the standard bounded-second moment assumption for stochastic weakly convex optimization \cite[Assumption A3]{davis_stochastic_2019}. In this section, we will study the stochastic subgradient method \cref{alg:ssgd} for solving problem \cref{eq:para-prob}.

\begin{algorithm}
\caption{Stochastic subgradient method}
\label{alg:ssgd}
\begin{algorithmic}[1]
\REQUIRE{Initial model $x_0$, step size sequence $\{\gamma_t\}_{t \in \N} \subset \R^+$}
\FOR{$t \in [0, T]$}
\STATE{Draw $\xi_t \overset{i.i.d.}{\sim} \pi$}
\STATE{Update $x_{t+1} = x_t - \gamma_t g(x_t, \xi_t)$}
\ENDFOR
\STATE{Sample $t^*$ according to $\mathbb{P}(t^*=t) = \gamma_t/\sum_{t'=0}^T \gamma_{t'}$}
\RETURN $x_{t^*}$
\end{algorithmic}
\end{algorithm}

The following basic inequalities will be used frequently in the analysis.

\begin{fact}\label{useful-ineq} For $0<\alpha\le 1,$ one has:
\begin{itemize}
    \item[(i)] (Bernoulli's inequality) For all $a>0$ and $b\ge -a,$
    \begin{equation*}\label{eq:useful_ineq_1}
        (a+b)^\alpha\le a^\alpha+\alpha ba^{\alpha-1}.
    \end{equation*}
    \item[(ii)] (Subadditivity of concave function) For all $a>0,b>0,$
    \begin{equation*}\label{eq:useful_ineq_2}
        (a+b)^\alpha\le a^\alpha+b^\alpha.
    \end{equation*}
\end{itemize}
\end{fact}

The following inequalities will also be useful to derive the convergence rates.
\begin{fact}\label{fact:sum_step_size} For $f: [a, b+1] \to \R_{+}$ non-increasing, the following holds:
    \begin{equation*}
        \int_{a}^{b+1} f(x) dx \leq \sum_{x=a}^b f(x) \leq f(a) + \int_{a}^b f(x) dx.
    \end{equation*}

    Specially, for $\beta > 0$, one has:
    \begin{itemize}
        \item[(i)] If $\beta > 1$, then
        \begin{equation*}
            \sum_{t=0}^T (t+1)^{-\beta} \leq \frac{\beta}{\beta-1}.
        \end{equation*}
        \item[(ii)] If $\beta = 1$, then
        \begin{equation*}
            \sum_{t=0}^T (t+1)^{-1} \leq \ln(T+1) + 1.
        \end{equation*}
        \item[(iii)] If $\beta < 1$, then
        \begin{equation*}
            \sum_{t=0}^T (t+1)^{-\beta} \geq \frac{1}{1-\beta} \pa{ (T+2)^{1-\beta} - 1 }
        \end{equation*}
    \end{itemize}
\end{fact}

Henceforth, we denote by $\E[t]{\cdot}$ the conditional expectation given the $\sigma$-algebra generated by the random samples $\xi_0, \xi_1, \dots, \xi_{t-1}$. Following \cref{thm:generalized_moreau_paraweak}, the near-stationarity measure used in this section will involve the generalized Moreau envelope and the generalized proximal mapping with the kernel $h(t) = (1/p)t^p$ for $t \geq 0$. A simple argument shows that $h$ satisfies \cref{eq:cond_h} with $c=2^{p-1}$. Thus, a regularization parameter $ 0 < r < (2^{p-1}\rho)^{-1}$ guarantees that $d\pa{ 0, \moreaup[r]{x} }$ can be used as the near-stationarity measure. To lighten the notation, when clear from context, we may write $g_t$ for $g(x_t,\xi_t)$. We begin with the following descent lemma.

\begin{lemma}\label{lem:unconstrained_descent_lem}
    Let $\{x_t\}$ be the sequence generated by \cref{alg:ssgd}, and for $0 < r < (2^{p-1}\rho)^{-1}$, let $z_t \in \gproxp[r]{x_t}$. Then, under \cref{assump:stochastic_grad_oracle}, one has:
    \begin{equation*}
        \E[t]{ \norm[p]{z_t - x_{t+1}} } \leq \norm[p]{z_t - x_t} - (r^{-1} - \rho) \gamma_t \norm[2(p-1)]{z_t - x_t} + \left( 2^{p-1} p + 1 \right) \gamma_t^p L^p
    \end{equation*}
\end{lemma}

\begin{proof}
    By the update rule of \cref{alg:ssgd}, one can simply deduce:
    \begin{align*}
        \norm[p]{z_t - x_{t+1}}
        &= \left( \norm[2]{z_t - x_t} + 2\gamma_t \langle z_t-x_t, g_t \rangle + \gamma_t^2 \norm[2]{g_t} \right)^{p/2}.
    \end{align*}

    We have the following two cases.

    \textbf{Case 1:} $ \norm[2]{z_t - x_t} + 2 \gamma_t \langle z_t-x_t, g_t \rangle > 0$. Then by using \cref{useful-ineq}, one successively deduce:
    \begin{align*}
        & \left( \norm[2]{z_t - x_t} + 2\gamma_t \langle z_t-x_t, g_t \rangle + \gamma_t^2 \norm[2]{g_t} \right)^{p/2} \\
        & \leq \left( \norm[2]{z_t - x_t} + 2 \gamma_t \langle z_t - x_t, g_t \rangle \right)^{p/2} + \gamma_t^p \norm[p]{g_t} \\
        & \leq \norm[p]{z_t-x_t} + p \gamma_t \langle z_t-x_t, g_t \rangle \norm[p-2]{z_t-x_t} + \gamma_t^p \norm[p]{g_t}.
    \end{align*}

    \textbf{Case 2:} $  \norm[2]{z_t-x_t} + 2\gamma_t \langle z_t-x_t, g_t \rangle \leq 0  $. From this condition, one has:
    \begin{align*}
        \norm[2]{z_t-x_t} \leq 2 \gamma_t \langle x_t-z_t, g_t \rangle \leq 2 \gamma_t \norm{z_t-x_t}\norm{g_t},
    \end{align*}
    which implies
    \begin{equation}\label{eq:descent_lem_unconstrained_proof_1}
        \norm{z_t-x_t} \leq 2 \gamma_t \norm{g_t}.
    \end{equation}

    Then, by the fact that $a \mapsto a^{p/2}$ is non-decreasing for $a \geq 0$ and $p \in (1, 2]$, one obtains:
    \begin{align*}
        & \left( \norm[2]{z_t - x_t} + 2\gamma_t \langle z_t-x_t, g_t \rangle + \gamma_t^2 \norm[2]{g_t} \right)^{p/2} \\
        & \leq \norm[p]{z_t-x_t} + p \gamma_t \langle z_t-x_t, g_t \rangle \norm[p-2]{z_t-x_t} + \gamma_t^p \norm[p]{g_t} \\
        & \quad - p \gamma_t \langle z_t-x_t, g_t \rangle \norm[p-2]{z_t-x_t}.
    \end{align*}

    By using \cref{eq:descent_lem_unconstrained_proof_1}, the last term above can be upper-bounded as:
    \begin{align*}
        &- p \gamma_t \langle z_t-x_t, g_t \rangle \norm[p-2]{z_t-x_t}
        \leq 2^{p-1}p \gamma_t^p \norm[p]{g_t},
    \end{align*}
    which gives
    \begin{align*}
        &\left( \norm[2]{z_t - x_t} + 2\gamma_t \langle z_t-x_t, g_t \rangle + \gamma_t^2 \norm[2]{g_t} \right)^{p/2} \\
        & \leq \norm[p]{z_t-x_t} + p \gamma_t \langle z_t-x_t, g_t \rangle \norm[p-2]{z_t-x_t} + (2^{p-1}p+1) \gamma_t^p \norm[p]{g_t}.
    \end{align*}

    Combining the two cases, taking the conditional expectation, together with using \cref{assump:stochastic_grad_oracle}, one has:
    \begin{align}
        \E[t]{ \norm[p]{z_t-x_{t+1}} }
        & \leq \norm[p]{z_t-x_t} + p \gamma_t \langle z_t-x_t, v_t \rangle \norm[p-2]{z_t-x_t} \nonumber \\
        & \quad + (2^{p-1}p+1) \gamma_t^p L^p,\label{eq:descent_lem_unconstrained_proof_2}
    \end{align}
    with $v_t \in \partial F(x_t)$. By $(\rho, p)$-paraconvexity of $F$, the second term in \cref{eq:descent_lem_unconstrained_proof_2} can be estimated as:
    \begin{subequations}
        \begin{align}
            & p \gamma_t \langle z_t-x_t, v_t \rangle \norm[p-2]{z_t-x_t} \nonumber \\
            & \leq p\gamma_t \left( F(z_t)-F(x_t) + \frac{\rho}{p} \norm[p]{z_t-x_t} \right) \norm[p-2]{z_t-x_t} \label{eq:descent_lem_unconstrained_proof_3} \\
            & \leq p \gamma_t \left( \frac{\rho}{p} - \frac{1}{rp} \right) \norm[p]{z_t-x_t} \norm[p-2]{z_t-x_t} \label{eq:descent_lem_unconstrained_proof_4} \\
            & = - (r^{-1} - \rho) \gamma_t \norm[2(p-1)]{z_t-x_t}. \label{eq:descent_lem_unconstrained_proof_5}
        \end{align}
    \end{subequations}

    \Cref{eq:descent_lem_unconstrained_proof_3} follows from \cref{charac-ineq1}; while \cref{eq:descent_lem_unconstrained_proof_4} follows from $z_t \in \gprox{x_t}$, which implies:
    \begin{align*}
        F(z_t) + \frac{1}{rp} \norm[p]{z_t-x_t} \leq F(x_t).
    \end{align*}

    Putting \cref{eq:descent_lem_unconstrained_proof_5} back into \cref{eq:descent_lem_unconstrained_proof_2} concludes the proof.
\end{proof}

We are now ready to state the following theorem on the rate of convergence of \cref{alg:ssgd} for solving problem \cref{eq:para-prob}.

\begin{theorem}\label{thm:ssgd} Under \cref{assump:stochastic_grad_oracle}, for $0 < r < (2^{p-1} \rho)^{-1}$, the sequence $\set{x_t}$ generated by \cref{alg:ssgd} satisfies:
\begin{align}\label{eq:unconstrained_conv_1}
&\E{ \sum_{t=0}^T \gamma_t d \left(0, \partial \moreaup[r]{x_t}\right)^2 } \nonumber \\
&\le \frac{p}{1-r\rho}  \pa{ \moreaup[r]{x_0} - \min F } + \frac{ \pa{ 2^{p-1}p+1 }L^p }{r(1-r\rho)} \sum_{t=0}^T \gamma_t^p.
\end{align}

In particular, if \cref{alg:ssgd} uses step size $\gamma_t = (t+1)^{-1/p}$, then its output $x_{t^*}$ will satisfy:
\begin{align}\label{eq:unconstrained_conv_2}
    & \E{ d \left( 0, \partial M^{1/2\rho}_F (x_{t^*}) \right)^2 } \nonumber \\
    & \leq c_1 ~ \frac{\moreaup[1/2\rho]{x_0}-\min F}{(T+2)^{1-1/p}-1} + c_2 ~ \frac{L^p (\ln\pa{T+1}+1)}{(T+2)^{1-1/p}-1},
\end{align}
with 
\begin{equation*}
    c_1 = 2 (p-1), \quad c_2 = \frac{ 4\rho (p-1) \pa{ 2^{p-1}p+1 } }{p}.
\end{equation*}
\end{theorem}

\begin{proof}
    For any $t \in [0, T]$, by \cref{thm:generalized_moreau_paraweak}~(iii), one can pick $z_t \in \gproxp[r]{x_t}$ such that
    \begin{equation}\label{eq:unconstrained_conv_proof_0}
        d \pa{ 0, \moreaup[r]{x_t} } = r^{-1} \norm[p-1]{x_t-z_t}.
    \end{equation}
    
    By definition, one successively deduces:
    \begin{subequations}
    \begin{align}
        & \E[t]{\moreaup[r]{x_{t+1}}} \nonumber \\
        & \leq F(z_t) + \frac{1}{rp} \E[t]{\norm[p]{z_t-x_{t+1}}} \label{eq:unconstrained_conv_proof_1} \\
        & \leq \moreaup[r]{x_t} - \frac{ r^{-1} - \rho }{rp} \gamma_t \norm[2(p-1)]{z_t-x_t}
        + \frac{\pa{ 2^{p-1}p+1 }L^p}{rp} \gamma_t^p \label{eq:unconstrained_conv_proof_2},
    \end{align}
    \end{subequations}
    where \cref{eq:unconstrained_conv_proof_1} follows directly from the definition of the generalized proximal mapping, and \cref{eq:unconstrained_conv_proof_2} follows from \cref{lem:unconstrained_descent_lem}. By rearranging, taking the full expectation, then taking the sum from $t=0$ to $T$, one arrives at:
    \begin{align}
        &\E{ \sum_{t=0}^{T} \gamma_t \norm[2(p-1)]{z_t-x_t} } \nonumber \\
        & \leq \frac{rp}{r^{-1}-\rho} \left(\moreaup[r]{x_0} - \min F \right)
        + \frac{\pa{ 2^{p-1}p+1 }L^p}{r^{-1}-\rho} \sum_{t=0}^T \gamma_t^p. \label{eq:unconstrained_conv_proof_3}
    \end{align}

    Noticing that, by \cref{eq:unconstrained_conv_proof_0}, one has:
    \begin{equation*}
        d\left(0, \partial \moreaup[r]{x_t}\right)^2 = r^{-2} \norm[2(p-1)]{z_t-x_t},
    \end{equation*}
    which, together with \cref{eq:unconstrained_conv_proof_3}, proves \cref{eq:unconstrained_conv_1}. Let $r = 1/2\rho$, then by dividing two sides of \cref{eq:unconstrained_conv_1} by $\sum_{t=0}^T \gamma_t$, one obtains:
    \begin{align*}
        &\E{ d \left( 0, \partial M^{1/2\rho}_F (x_{t^*}) \right)^2 } \\
        & \leq \frac{2p \pa{ \moreaup[r]{x_0} - \min F } }{\sum_{t=0}^T \gamma_t} + 4 \rho \pa{ 2^{p-1}p+1 } L^p \frac{ \sum_{t=0}^T \gamma_t^p }{\sum_{t=0}^T \gamma_t}.
    \end{align*}

    Using \cref{fact:sum_step_size} gives:
    \begin{align*}
        & \sum_{t=0}^T \gamma_t = \sum_{t=0}^T (t+1)^{-1/p} \geq \frac{p}{p-1} \pa{ (T+2)^{1-1/p}-1 },\\ 
        & \sum_{t=0}^T \gamma_t^p = \sum_{t=0} (t+1)^{-1} \leq \ln\pa{T+1} + 1,
    \end{align*}
    which proves \cref{eq:unconstrained_conv_2}.
\end{proof}

\begin{remark}\label{remark:thm_unconstrained}
    \cref{thm:ssgd} guarantees the convergence rate $\mathcal{O}\pa{\ln T/T^{1-1/p}}$ for \cref{alg:ssgd} with a versatile decaying step size rule $\gamma_t = (t+1)^{-1/p}$ does not require any additional knowledge about the optimization problem except for $p$. Using a constant step size $\gamma_t \equiv \gamma = cT^{-1/p}$, then optimizing for $c$ can remove the logarithmic factor and yields smaller constants $c_1, c_2$. However, this requires knowing in advance the horizon $T$ and other problem-dependent constants. In practice, one can use $\gamma_t = c(t+1)^{-1/p}$ then tunes for $c$.
\end{remark}

\section{Stochastic proximal subgradient method for composite paraconvex optimization}\label{sec:proximal}
Consider the problem of \textit{composite paraconvex optimization:}
\begin{equation}\label{CP-prob}
\min_{x\in\R^n}F(x)\coloneqq f(x)+\varphi(x),
\end{equation}
where $\varphi:\R^n\to\R\cup\{+\infty\}$ is a proper l.s.c. convex function, and $f:\R^n\to\R\cup\{+\infty\}$ is proper, l.s.c. and $(\rho, p)$-paraconvex. We also make the following assumption on the stochastic gradient oracle that is similar to \cref{assump:stochastic_grad_oracle}.

\begin{assumption}\label[assumption]{assump:stochastic_grad_oracle_1}
    Fix a probability space $\left( \Omega, \mathcal{F}, \pi \right)$, and equip $\R^n$ with the Borel $\sigma$-algebra. We make the following assumptions:
    \begin{itemize}
    \item[(i)] We can draw i.i.d. realizations $\xi_1, \xi_2, \dots$ from $\pi$.
    \item[(ii)] There exists an open set $U \supset \dom{\varphi}$ on which $f$ is finite, and a measurable mapping $g: U \times \Omega \rightarrow\R^n$ such that $\mathbb{E}_{\xi\sim \pi} g(x,\xi)\in\partial F(x),$ for all $x\in U$.
    \item[(iii)] There exists a constant $L>0$ such that $\mathbb{E}_{\xi\sim P} \|g(x,\xi)\|^p\le L^p,$ for all $x \in \dom{\varphi}$.
    \end{itemize}
\end{assumption}


Noticing that, similarly as in \cite{davis_stochastic_2019}, \cref{assump:stochastic_grad_oracle_1}~(iii) implies $\|v\|\le L$ for all $v\in\partial f(x),$ $x\in\dom\varphi$. We will investigate the proximal stochastic subgradient method (\cref{alg:psgd}) for solving problem \cref{CP-prob} with $p \in (1, 2)$ for simplicity. The convergence rate of \cref{alg:psgd} for the weakly convex case $p=2$ given in \cite{davis_stochastic_2019} can be recovered with an additional assumption on the step size. We provide more discussion in \cref{remark:recover_weakly_convex}. 
\begin{algorithm}
\caption{Proximal stochastic subgradient method}
\label{alg:psgd}
\begin{algorithmic}[1]
\REQUIRE{Initial model $x_0 \in \dom{\varphi}$, step size sequence $\{\gamma_t\}_{t \in \N} \subset \R^+$}
\FOR{$t \in [0, T]$}
\STATE{Draw $\xi_t \overset{i.i.d.}{\sim} \pi$}
\STATE\label{line:proximal_update}{Update $x_{t+1} = \prox[\gamma_t \varphi]{ x_t - \gamma_t g(x_t, \xi_t) }$}
\ENDFOR
\STATE{Sample $t^*$ according to $\mathbb{P}(t^*=t) = \gamma_t/\sum_{t'=0}^T \gamma_{t'}$}
\RETURN $x_{t^*}$
\end{algorithmic}
\end{algorithm}

In \cref{alg:psgd}, the update is performed using the standard proximal operator with respect to the convex function $\varphi$ and smoothing parameter $\gamma_t > 0$:
\begin{equation*}
    \prox[\gamma_t \varphi]{ x } = \margmin{ \varphi(y) + \frac{1}{2\gamma_t} \norm[2]{y-x} }.
\end{equation*}

Before going to the analysis, we introduce the following useful inequalities, which will be used frequently.
\begin{fact}\label{useful-ineq1}
    For $a, b \in \R^n$, $p > 1$, the following holds:
    \begin{itemize}
        \item[(i)] For every $\lambda \in (0, 1)$, $ \norm[p]{a+b} \geq \lambda^{p-1} \norm[p]{a} - \pa{ \frac{\lambda}{1-\lambda} }^{p-1} \norm[p]{b} $.
        \item[(ii)] $ \norm[p]{a-b} \leq 2^{p-1} \pa{ \norm[p]{a} + \norm[p]{b} } $.
    \end{itemize}
\end{fact}

\begin{proof}
    The first statement follows from the convexity of $\norm[p]{.}$ and from noticing that $a = \lambda \frac{a+b}{\lambda} + (1-\lambda) \frac{-b}{1-\lambda} $. The second statement follows by setting $\lambda=1/2$ and substituting $a-b$ for $a$ in the first statement.
\end{proof}

Similarly to \cref{sec:ssgd}, the analysis will use the generalized Moreau envelope and the generalized proximal mapping with the kernel $h(t)=(1/p)t^p$. We break the analysis of \cref{alg:psgd} into two lemmas. In the following, fix $0 < r < (2\rho)^{-1}$, and let $z_t = \gproxp[r]{x_t}$. The first lemma relates $z_t$ as a proximal point of $\varphi$ using the optimality conditions of the (generalized) proximal mappings and the subdifferential sum rule.
\begin{lemma}\label{lem:relating_2_proximal}
    For every $t \in [0, T]$, there exists $\hat{v}_t \in \partial f(z_t)$ such that:
    \begin{subequations}
    \begin{align}
        z_t &= \prox[\gamma_t \varphi]{y_t},\label{eq:prox_z_t}\\
        y_t &= z_t - r^{-1}\gamma_t \norm[p-2]{z_t-x_t}(z_t-x_t) - \gamma_t \hat{v}_t.\label{eq:y_t}
    \end{align}
    \end{subequations}
\end{lemma}

\begin{proof}
    Since $f$ is finite on an open set containing $\dom{\varphi}$, and $z_t \in \dom{\varphi}$ by the definition of $z_t$, the subdifferential sum rule gives:
    \begin{equation*}
        \partial F(z_t) = \partial f(z_t) + \partial \varphi(z_t).
    \end{equation*}

    Hence, by the optimality condition of the generalized proximal mapping, there exists $\hat{v}_t \in \partial f(z_t)$ such that:
    \begin{equation*}
        0 \in v_t + \partial \varphi(z_t) + r^{-1}\norm[p-2]{z_t-x_t}(z_t-x_t).
    \end{equation*}

    Let $y_t$ be defined by \cref{eq:y_t}, then equivalently, one has:
    \begin{equation*}
        z_t - y_t \in \gamma_t \partial \varphi(z_t),
    \end{equation*}
    which gives \cref{eq:prox_z_t} by the optimality condition for the proximal mapping.
\end{proof}

The following lemma establishes a descent property for the iterates generated by the stochastic proximal subgradient method.
\begin{lemma}\label{lem:cp_descent_lem}
    Given \cref{assump:stochastic_grad_oracle_1}, the sequence $\set{x_t}$ generated by \cref{alg:psgd} satisfies:
    \begin{align*}
        & \E[t]{\norm[p]{z_t-x_{t+1}}} \\
         & \leq \norm[p]{z_t-x_t} - (r^{-1}-2\rho) \gamma_t \norm[2(p-1)]{z_t-x_t}
         + c_3 \gamma_t^{\frac{p}{2-p}}
         + c_4 \gamma_t^p L^p.
    \end{align*}
    with
    \begin{equation*}
        c_3 = \pa{ 2^{p-1}r^{-\frac{p}{2-p}} + (r^{-1}-2\rho) r^{-\frac{2(p-1)}{2-p}} }, \quad 
        c_4 = 2^p \pa{1+2^{p-1}p}.
    \end{equation*}
\end{lemma}

\begin{proof}
     Let $\hat{v}_t \in \partial f(z_t)$. Then, by using \cref{lem:relating_2_proximal} and the update rule of \cref{alg:psgd}, together with the non-expansiveness of $\operatorname{prox}_{\gamma_t \varphi}(.)$, one successively deduces:
     \begin{subequations}
     \begin{align}
         &\norm[p]{z_t - x_{t+1}} \nonumber \\
         &= \norm[p]{\prox{z_t - r^{-1}\gamma_t \norm[p-2]{z_t-x_t}(z_t-x_t) - \gamma_t \hat{v}_t} - \prox{x_t - \gamma_t g_t}} \nonumber \\
         & \leq \norm[p]{ \alpha_t (z_t-x_t) - \gamma_t (\hat{v}_t-g_t) } \label{eq:descent_lem_cp_proof} \\
         & \leq \pa{ \alpha_t^2 \norm[2]{z_t-x_t} - 2\gamma_t\alpha_t \langle z_t-x_t, \hat{v}_t-g_t \rangle + \gamma_t^2 \norm[2]{\hat{v_t}-g_t} }^{p/2},\label{eq:descent_lem_cp_proof_1}
     \end{align}
     \end{subequations}
     with $\alpha_t = 1-r^{-1}\gamma_t \norm[p-2]{z_t-x_t} $. We first prove the following claim:

     \textbf{Claim 1:} If $\alpha_t > 0$, then:
     \begin{align*}
         & \E[t]{ \norm[p]{z_t-x_{t+1}} } \\
         & \leq \norm[p]{z_t-x_t} - (r^{-1}-2\rho) \gamma_t \norm[2(p-1)]{z_t-x_t} + 2^p (1+2^{p-1}p) \gamma_t^p L^p.
     \end{align*}

     To prove \textbf{Claim 1}, we need to distinguish the following two cases.

     \textbf{Case 1:} $\alpha_t^2 \norm[2]{z_t-x_t} - 2 \gamma_t \alpha_t \inn{ z_t-x_t, \hat{v}_t - g_t } > 0 $. Then, from \cref{eq:descent_lem_cp_proof_1}, by applying successively the inequalities in \cref{useful-ineq}, one obtains:
     \begin{align}
         & \norm[p]{z_t-x_{t+1}} \nonumber \\
         & \leq \alpha_t^p \norm[p]{z_t-x_t} + p\gamma_t \alpha_t^{p-1} \inn{ z_t-x_t, g_t-\hat{v}_t } \norm[p-2]{z_t-x_t}
          + \gamma_t^p { \norm[p]{\hat{v}_t-g_t} } .\label{eq:descent_lem_cp_proof_2}
     \end{align}

     \textbf{Case 2:} $\alpha_t^2 \norm[2]{z_t-x_t} - 2\gamma_t \alpha_t \inn{ z_t-x_t, \hat{v}_t-g_t } \leq 0$. This condition gives:
     \begin{equation}\label{eq:descent_lem_cp_proof_3}
        \alpha_t \norm{z_t-x_t} \leq 2\gamma_t \norm{\hat{v}_t-g_t}.
     \end{equation}

     By using the fact that $a \mapsto a^{p/2}$ is non-increasing for $a \geq 0$ and $p \in (1, 2]$, from \cref{eq:descent_lem_cp_proof_1}, one arrives at:
     \begin{align}
         & \norm[p]{z_t-x_{t+1}} \nonumber \\
         & \leq \alpha_t^p \norm[p]{z_t-x_t} + p\gamma_t \alpha_t^{p-1} \inn{ z_t-x_t, g_t-\hat{v}_t } \norm[p-2]{z_t-x_t}
          + \gamma_t^p { \norm[p]{\hat{v}_t-g_t} } \nonumber \\
         & \quad + p\gamma_t \alpha_t^{p-1} \inn{ z_t-x_t, \hat{v_t}-g_t } \norm[p-2]{z_t-x_t} \nonumber \\
         & \leq \alpha_t^p \norm[p]{z_t-x_t} + p\gamma_t \alpha_t^{p-1} \inn{ z_t-x_t, g_t-\hat{v}_t } \norm[p-2]{z_t-x_t} \nonumber \\
         & \quad + \pa{ 1+2^{p-1}p } \gamma_t^p { \norm[p]{\hat{v}_t-g_t} }, \label{eq:descent_lem_cp_proof_4}
     \end{align}
     where the last inequality follows from \cref{eq:descent_lem_cp_proof_3}. Combining \eqref{eq:descent_lem_cp_proof_2} and \eqref{eq:descent_lem_cp_proof_4}, then taking the conditional expectation, one obtains:
     \begin{align}
         & \E[t]{\norm[p]{z_t-x_{t+1}}} \nonumber \\
         & \leq \alpha_t^p \norm[p]{z_t-x_t} + p\gamma_t \alpha_t^{p-1} \inn{ z_t-x_t, \EE_t g_t-\hat{v}_t } \norm[p-2]{z_t-x_t} \nonumber \\
         & \quad + \pa{ 1+2^{p-1}p } \gamma_t^p \E[t]{ \norm[p]{\hat{v}_t-g_t} }. \label{eq:descent_lem_cp_proof_5}
     \end{align}

     By \cref{useful-ineq1}, the last term in \cref{eq:descent_lem_cp_proof_5} can be evaluated as:
     \begin{equation}
         \E[t]{ \norm[p]{\hat{v}_t -g_t} } \leq 2^{p-1} \pa{ \norm[p]{\hat{v}_t} + \E[t]{ \norm[p]{g_t} } }
         \leq 2^p L^p, \label{eq:descent_lem_cp_proof_9}
     \end{equation}
     which follows from \cref{assump:stochastic_grad_oracle_1}~(iii).

     The second term of \cref{eq:descent_lem_cp_proof_5} can be evaluated by noticing that $\EE_t g_t \in \partial f(x_t)$, hence
     \begin{equation*}
         \inn{ z_t-x_t, \EE_t g_t - \hat{v}_t } \leq \frac{2\rho}{p} \norm[p]{z_t-x_t},
     \end{equation*}
     which follows from \cref{charac-ineq2} by the paraconvexity of $f$. Finally, since $0 < \alpha_t \leq 1$, one has $\alpha_t^p \leq \alpha_t$, and $\alpha_t^{p-1} \leq 1$. Thus, \cref{eq:descent_lem_cp_proof_5} will lead to:
     \begin{align*}
         & \E[t]{ \norm[p]{z_t-x_{t+1}} } \nonumber \\
         & \leq \norm[p]{z_t-x_t} - (r^{-1}-2\rho) \gamma_t \norm[2(p-1)]{z_t-x_t} + 2^p (1+2^{p-1}p) \gamma_t^p L^p,
     \end{align*}
     which is exactly \textbf{Claim 1}. For the case where $\alpha_t \le 0$, we have the following claim.

     \textbf{Claim 2.} If $\alpha_t \le 0$, then
     \begin{align*}
         & \E[t]{\norm[p]{z_t-x_{t+1}}} \\
         & \leq \norm[p]{z_t-x_t} - (r^{-1}-2\rho) \gamma_t \norm[2(p-1)]{z_t-x_t} \\
         &\quad + \pa{ 2^{p-1}r^{-\frac{p}{2-p}} + (r^{-1}-2\rho) r^{-\frac{2(p-1)}{2-p}} } \gamma_t^{\frac{p}{2-p}}
         + 2^{2p-1}\gamma_t^p L^p.
     \end{align*}

     To prove \textbf{Claim 2}, first noticing that $\alpha_t \le 0$ gives:
     \begin{equation}
         \norm[2-p]{z_t-x_t} \leq r^{-1} \gamma_t.\label{eq:descent_lem_cp_proof_6}
     \end{equation}

     Then, from \cref{eq:descent_lem_cp_proof}, by using \cref{useful-ineq1}, one has:
     \begin{subequations}
     \begin{align}
         &\norm[p]{z_t-x_{t+1}}
         \leq 2^{p-1} \pa{ |\alpha_t|^p \norm[p]{z_t-x_t} + \gamma_t^p \norm[p]{\hat{v}_t-g_t} } \nonumber \\
         & \leq \norm[p]{z_t-x_t} - (r^{-1}-2\rho)\gamma_t \norm[2(p-1)]{z_t-x_t} \nonumber \\
         & \quad + 2^{p-1} \pa{ \norm[p]{z_t-x_t} + \gamma_t^p \norm[p]{\hat{v}_t-g_t} } 
         + (r^{-1}-2\rho)\gamma_t \norm[2(p-1)]{z_t-x_t} \label{eq:descent_lem_cp_proof_7} \\
         & \leq \norm[p]{z_t-x_t} - (r^{-1}-2\rho) \gamma_t \norm[2(p-1)]{z_t-x_t} \nonumber \\
         & \quad + 2^{p-1} \pa{ r^{ -\frac{p}{2-p} }\gamma_t^{\frac{p}{2-p}} + \gamma_t^p \norm[p]{\hat{v}_t-g_t} }
         + (r^{-1}-2\rho) r^{ -\frac{2(p-1)}{2-p} } \gamma_t ^{ \frac{p}{2-p}}, \label{eq:descent_lem_cp_proof_8}
     \end{align}
     \end{subequations}
     where \cref{eq:descent_lem_cp_proof_7} follows from $|\alpha_t| \leq 1$, and \cref{eq:descent_lem_cp_proof_8} follows from \cref{eq:descent_lem_cp_proof_6}. Taking the conditional expectation, then using \cref{eq:descent_lem_cp_proof_9} proves \textbf{Claim 2}. Combining the results from two claims proves the lemma.
\end{proof}

Now we are ready to state the following theorem on the rate of convergence of \cref{alg:psgd} for solving problem \cref{CP-prob}.

\begin{theorem}\label{thm:conv_psgd}
    Under \cref{assump:stochastic_grad_oracle_1}, for $0 < r < (2\rho)^{-1}$, the sequence $\set{x_t}$ generated by \cref{alg:psgd} satisfies:
    \begin{align}\label{eq:cp_conv_1}
    &\E{ \sum_{t=0}^T \gamma_t d \left(0, \partial \moreaup[r]{x_t}\right)^2 } \nonumber \\
    &\le \frac{p}{1-2\rho r}  \pa{ \moreaup[r]{x_0} - \min F } + \frac{ c_3 }{r(1-2\rho r)} \sum_{t=0}^T \gamma_t^{\frac{p}{2-p}} 
    + \frac{c_4L^p}{r(1-2\rho r)} \sum_{t=0}^T \gamma_t^p.
    \end{align}
    
    In particular, if \cref{alg:psgd} uses step size $\gamma_t = (t+1)^{-1/p}$, then its output $x_{t^*}$ will satisfy:
    \begin{align}\label{eq:cp_conv_2}
        & \E{ d \left( 0, \partial M^{1/4\rho}_F (x_{t^*}) \right)^2 } \nonumber \\
        & \leq \frac{c_5~(\moreaup[1/4\rho]{x_0}-\min F) + c_6}{(T+2)^{1-1/p}-1} + \frac{c_7 L^p (\ln\pa{T+1}+1)}{(T+2)^{1-1/p}-1},
    \end{align}
    with 
    \begin{equation*}
        c_5 = 2(p-1), ~ c_6 = \frac{8 \rho^{\frac{2}{2-p}}}{p} \pa{ 2^{p-1+\frac{2p}{2-p}} + 2^{\frac{3p-2}{2-p}} }, ~
        c_7 = \frac{8\rho (p-1) c_4}{p},
    \end{equation*}
    $c_3$ and $c_4$ defined in \cref{lem:cp_descent_lem}.
\end{theorem}

\begin{proof}
    For any $t \in [0, T]$, by \cref{thm:generalized_moreau_paraweak}~(iii), one can pick $z_t \in \gproxp[r]{x_t}$ such that
    \begin{equation}\label{eq:cp_conv_proof_0}
        d \pa{ 0, \moreaup[r]{x_t} } = r^{-1} \norm[p-1]{x_t-z_t}.
    \end{equation}
    
    By definition of the generalized Moreau envelope, together with \cref{lem:cp_descent_lem}, one directly obtains, after rearranging, taking the full expectation, and taking the sum from $t=0$ to $T$:
    \begin{align*}
        & \frac{1-2\rho r}{r^2 p} \E{ \sum_{t=0}^T \gamma_t \norm[2(p-1)]{z_t-x_t} } \nonumber \\
        & \leq \moreaup[r]{x_0} - \min F + \frac{c_3}{rp} \sum_{t=0}^T \gamma_t ^{\frac{p}{2-p}} + \frac{c_4 L^p}{rp} \sum_{t=0}^T \gamma_t^P
    \end{align*}

    Then, using \cref{eq:cp_conv_proof_0} gives \cref{eq:cp_conv_1}. Dividing both sides of \cref{eq:cp_conv_1} by $\sum_{t=0}^T \gamma_t$, letting $r=1/4\rho$, then using the bounds on the sum of step size at the end of the proof of \cref{thm:ssgd}, together with:
    \begin{align*}
        \sum_{t=0}^T \gamma_t^{\frac{p}{2-p}} = \sum_{t=0}^T (t+1)^{-\frac{1}{2-p}}
        \leq \frac{1}{p-1}
    \end{align*}
    proves \cref{eq:cp_conv_2}.
\end{proof}

\begin{remark}\label{remark:recover_weakly_convex}
    Some remarks are as follows:
    \begin{itemize}
        \item[(i)] With the use of a constant step size $\gamma_t \equiv \gamma = cT^{-1/p}$, one can derive \cref{eq:cp_conv_2} with smaller constants and without the logarithmic factor, given the knowledge of the horizon $T$ and other problem-dependent constants.
        \item[(ii)] The convergence rate established in \cref{thm:conv_psgd} is robust to the choice of step size, in the sense that no condition on the step size is imposed, but in contrast, we need to restrict $p$ to be in $(1, 2)$. In the weakly convex case $p=2$, one can notice that the constants $c_3$ and $c_6$ tend to $+\infty$, while $\sum_{t=0}^T \gamma_t^{p/(2-p)} \to 0$. Hence, if one adopts the convention $0.(+\infty)=0$, the rate $\mathcal{O}(T^{-1/2})$ for the weakly convex case, as in \cite{davis_stochastic_2019}, can be recovered. Another way to recover the weakly convex rate is to impose the assumption $\gamma_t/r \leq 1$, which is similar to the assumption imposed in \cite[Theorem 3.4]{davis_stochastic_2019}. With this assumption, the case $\alpha_t \leq 0$ in the proof of \cref{lem:cp_descent_lem} will not happen if $p=2$, hence the second term in \cref{eq:cp_conv_1} will not appear.
    \end{itemize}
\end{remark}

\section{Stochastic model-based paraconvex optimization} The previous section establishes the convergence of the stochastic proximal subgradient methods for minimizing a $(\rho, p)$-paraconvex function. In this section, we study the stochastic model-based framework that encompasses a wider class of algorithms for minimizing the class of $(\rho, \eta, p)$-paraconvex functions. We will consider the problem
\begin{equation}\label{eq:cp_prob_model_based}
\min_{x\in\R^n} F(x)\coloneqq f(x) + \varphi(x),
\end{equation}
where $\varphi: \R^n \to \R \cup \set{ +\infty }$ is a proper l.s.c. function (not necessarily convex), and $f: \R^n \to \R \cup \set{+\infty}$ a proper l.s.c. function. We assume that the access to $f$ is through the following \emph{stochastic one-sided model}.

\begin{assumption}[stochastic one-sided model]\label[assumption]{assump:stochastic_one_sided_model}
    Fix a probability space $\left( \Omega, \mathcal{F}, \pi \right)$, and equip $\R^n$ with the Borel $\sigma$-algebra. We assume that there exist positive reals $p \in (1, 2], \nu \in (0, 1], \tau, \eta, L_1, L_2 \in \R^+$ such that the followings hold:
    \begin{itemize}
    \item[(i)] We can draw i.i.d. realizations $\xi_1, \xi_2, \dots$ from $\pi$.
    \item[(ii)] There exists an open set $U \supset \dom{\varphi}$ on which $f$ is finite, and a measurable function $(x, y, \xi) \mapsto f_x(y, \xi)$, defined on $U \times U \times \Omega$ satisfying
    \begin{align}
        \E[\xi]{ f_x(x, \xi) } &= f(x) \quad \forall x \in U, \label{eq:stochastic_one_sided_model_1} \\
        \E[\xi]{ f_x(y, \xi) - f(y) } & \leq \frac{\tau}{p} \norm[p]{x-y} \quad \forall x,y \in U. \label{eq:stochastic_one_sided_model_2}
    \end{align}
    \item[(iii)] For all $x \in U$, and a.e. $\xi \sim \pi$, the function $f_x(., \xi) + \varphi(.)$ is $\eta$-weakly convex.
    \item[(iv)] There exists measurable functions $\mathcal{L}_1, \mathcal{L}_2: \Omega \to \R^+$ such that, for all $x, y \in U$, and a.e. $\xi \sim \pi$:
    \begin{align}
        & \pa{ \E[\xi]{ \mathcal{L}_1(\xi)^2} }^{1/2} \leq L_1,\quad \pa{ \E[\xi]{\mathcal{L}_2(\xi)^{\nu+1}} }^{1/(\nu+1)} \leq L_2, \nonumber \\
        & f_x(x, \xi) - f_x(y, \xi)  \leq \mathcal{L}_1(\xi) \norm[]{x-y} + \mathcal{L}_2 (\xi) \norm[\nu]{x-y} . \label{eq:stochastic_one_sided_model_3}
    \end{align}
    \end{itemize}
\end{assumption}

Our assumptions on the stochastic one-sided model are weaker than those of \cite[Assumption B]{davis_stochastic_2019} in three aspects: (i) we do not require $f$ to be locally Lipschitz; (ii) \cref{eq:stochastic_one_sided_model_2} permits a weaker approximation guarantee than \cite[Assumption B2]{davis_stochastic_2019}; and (iii) \cref{eq:stochastic_one_sided_model_3} does not necessarily require the model to be Lipschitz in the sense of \cite[Assumption B4]{davis_stochastic_2019}.


The following lemma establishes the paraconvexity of the objective function $F$, and the continuity of $f$ over $U$ under \cref{assump:stochastic_one_sided_model}.
\begin{lemma}\label{lem:model_based_lem_1}
    Under \cref{assump:stochastic_one_sided_model}, $F$ is $(2\tau, \eta, p)$-paraconvex, and, for all $x, y \in U$,
    \begin{equation}\label{eq:model_based_lipschitz}
        | f(x) -f(y) | \leq L_1 \norm[]{x-y} + L_2 \norm[\nu]{x-y} + \frac{\tau}{p} \norm[p]{x-y}.
    \end{equation}
\end{lemma}

\begin{proof}
    Fix arbitrary $x, y \in \dom \varphi$, and a real $t \in [0, 1]$, then let $\hat{x} = tx+(1-t)y$. Using the definition of paraconvexity, together with \cref{assump:stochastic_one_sided_model}~(ii), one successively deduce:
    \begin{subequations}
        \begin{align}
            &F(\hat{x}) = f(\hat{x}) + \varphi(\hat{x}) \nonumber \\
            &= \E[\xi]{ f_{\hat{x}}(\hat{x}, \xi) + \varphi(\hat{x}) } \label{eq:model_based_lem_1_proof_1} \\
            & \leq t \E[\xi]{ f_{\hat{x}}(x, \xi) + \varphi(x) } 
            + (1-t) \E[\xi]{ f_{\hat{x}}(y, \xi) + \varphi(y) }
            + \frac{\eta t(1-t)}{2} \norm[2]{x-y} \label{eq:model_based_lem_1_proof_2} \\
            & \leq t F(x) + (1-t) F(y) + \frac{ \tau \pa{ (1-t)^p t + (1-t)t^p } }{p} \norm[p]{x-y} \nonumber \\ 
            &\quad + \frac{\eta t(1-t)}{2} \norm[2]{x-y} \label{eq:model_based_lem_1_proof_3} \\
            & \leq t F(x) + (1-t) F(y) + \frac{2\tau t(1-t)}{p} \norm[p]{x-y} + \frac{\eta t(1-t)}{2} \norm[2]{x-y} \label{eq:model_based_lem_1_proof_4},
        \end{align}
    \end{subequations}
    where \cref{eq:model_based_lem_1_proof_1} uses \cref{eq:stochastic_one_sided_model_1}~(ii), \cref{eq:model_based_lem_1_proof_2} follows from the weak convexity of $f_{\hat{x}}(., \xi) + \varphi(.)$, \cref{eq:model_based_lem_1_proof_3} uses \cref{eq:stochastic_one_sided_model_2}, and \cref{eq:model_based_lem_1_proof_4} uses $(1-t)^p \leq 1-t$ and $t^p \leq t$ for $t \in [0, 1]$ and $p > 1$. To prove the second statement, taking expectation of \cref{eq:stochastic_one_sided_model_2} and \cref{eq:stochastic_one_sided_model_3} gives:
    \begin{align*}
        \E[\xi]{ f_x(y, \xi) } - f(y) \leq \frac{\tau}{p} \norm[p]{x-y}, \quad 
        f(x) - \E[\xi]{f_x(y, \xi)} \leq L_1 \norm[]{x-y} + L_2 \norm[\nu]{x-y} .
    \end{align*}

    Adding the two above inequalities yields the desired result.
\end{proof}

\begin{algorithm}
\caption{Stochastic model-based method}
\label{alg:model_based}
\begin{algorithmic}[1]
\REQUIRE{Initial model $x_0 \in \dom{\varphi}$, $\brho > 2(2\tau+\eta)$, a sequence $\{\beta_t\}_{t \in \N} \subseteq (\trho, +\infty) $, $\trho$ in \cref{eq:model_based_constant_6}}
\FOR{$t \in [0, T]$}
\STATE{Draw $\xi_t \overset{i.i.d.}{\sim} \pi$}
\STATE{Update $x_{t+1} = \margmin{ \varphi(x) + f_{x_t}(x, \xi_t) + \frac{\beta_t}{2} \norm[2]{x-x_t} } $}
\ENDFOR
\STATE{Sample $t^*$ according to $\mathbb{P}(t^*=t) \propto \frac{\min \set{ \brho-\eta,~ \frac{2(\brho-\tau)}{p} }}{\beta_t-\eta}$}
\RETURN $x_{t^*}$
\end{algorithmic}
\end{algorithm}

\subsection{Convergence analysis}

The optimization algorithm studied in this section is based on the stochastic one-sided model satisfying \cref{assump:stochastic_one_sided_model}, and is presented in \cref{alg:model_based}. The analysis relies on the generalized Moreau envelope and the generalized proximal mapping with the kernel $h(t) = (1/p)t^p + (1/2)t^2$ satisfying the condition \cref{eq:cond_h} with $c = 2$. Henceforth, let $z_t = \gproxp[1/\bar\rho][F]{x_t}$ with $\bar\rho > 2(2\tau+\eta)$. The convergence rate is established based on the following two descent lemmas.

\begin{lemma}
    Under \cref{assump:stochastic_one_sided_model}, with constants $A_{\brho, \eta}$ and $B_{\brho, \tau, p}$ defined in \cref{eq:model_based_constant_1}, $ C_{\brho, \tau, p}$, $D$ and $ E_{\nu} $ defined in \cref{eq:model_based_constant_2,eq:model_based_constant_3,eq:model_based_constant_4}, and
    \begin{align}
        \trho = \begin{cases}
            \brho & \text{if} ~ p \in (1, 2), \\
            2\brho+\tau & \text{if} ~ p = 2,
        \end{cases} \label{eq:model_based_constant_6}
    \end{align}
    the following holds:
    \begin{align*}
        & \E[t]{\norm[2]{z_t-x_{t+1}}} \nonumber \\
        & \leq \norm[2]{z_t-x_t} - (\beta_t-\eta)^{-1} \left( A_{\brho, \eta} \norm[2]{z_t-x_t} + B_{\brho, \tau, p} \norm[p]{z_t-x_t} \right) \nonumber \\
        & \quad + (\beta_t-\eta)^{-1} \pa{ C_{\brho, \tau, p} (\beta_t-\trho)^{-\frac{p}{2-p}} + D L_1^2 (\beta_t-\trho)^{-1} + E_{\nu} L_2^{\frac{2}{2-\nu}} (\beta_t-\trho)^{-\frac{\nu}{2-\nu}} }.
    \end{align*}
\end{lemma}

\begin{proof}
    First, by noticing that the function $x \mapsto \model[x_t]{x} + \varphi(x) + \frac{\beta_t}{2}\norm[2]{x-x_t} $ is $(\beta_t - \eta)$-strongly convex with $x_{t+1}$ its minimizer:
    \begin{align*}
        & \model[x_t]{x_{t+1}} + \varphi(x_{t+1}) + \frac{\beta}{2}\norm[2]{x_{t+1}-x_t} + \frac{\beta_t - \eta}{2}\norm[2]{z_t-x_{t+1}} \nonumber \\
        & \leq \model[x_t]{z_t} + \varphi(z_t) + \frac{\beta_t}{2}\norm[2]{z_t-x_t}.
    \end{align*}

    Rearranging and taking the conditional expectation, one successively deduces:
    \begin{subequations}
        \begin{align}
            & \E[t]{ \frac{\beta_t-\eta}{2}\norm[2]{z_t-x_{t+1}} + \frac{\beta_t}{2}\norm[2]{x_{t+1}-x_t} - \frac{\beta_t}{2} \norm[2]{z_t-x_t} } \nonumber \\
            & \leq \E[t]{ \model[x_t]{z_t} + \varphi(z_t) - \model[x_t]{x_{t+1}} - \varphi(x_{t+1}) } \nonumber \\
            & \leq F(z_t) - \E[t]{ \model[x_t]{x_{t+1}} + \varphi(x_{t+1}) } + \frac{\tau}{p}\norm[p]{z_t-x_t} \label{eq:model_based_descent_lem_1} \\
            & \leq F(z_t) - \E[t]{ \model[x_t]{x_t} + \varphi(x_{t+1}) } + \frac{\tau}{p} \norm[p]{z_t-x_t} \nonumber \\
            & \quad + \E[t]{ \mathcal{L}_1 (\xi_t) \norm[]{x_{t+1}-x_t} + \mathcal{L}_2(\xi_t) \norm[\nu]{x_{t+1}-x_t} } \label{eq:model_based_descent_lem_2} \\
            & \leq \E[t]{F(z_t) - F(x_{t+1}} + \frac{\tau}{p} \norm[p]{z_t-x_t} \nonumber \\
            & \quad + \E[t]{ \mathcal{L}_1 (\xi_t) \norm[]{x_{t+1}-x_t} + \mathcal{L}_2(\xi_t) \norm[\nu]{x_{t+1}-x_t} } \nonumber \\
            & \quad + \E[t]{ L_1 \norm[]{x_{t+1}-x_t} + L_2 \norm[\nu]{x_{t+1}-x_t} + \frac{\tau}{p} \norm[p]{x_{t+1}-x_t} } \label{eq:model_based_descent_lem_3} \\
            & \leq \E[t]{F(z_t) - F(x_{t+1})} + \frac{\tau}{p} \norm[p]{z_t-x_t} + \frac{\tau}{p} \E[t]{\norm[p]{x_{t+1}-x_t}} \nonumber \\
            & \quad + L_1 \pa{ \pa{ \E[t]{ \norm[2]{x_{t+1}-x_t} } }^{1/2} + \E[t]{\norm[]{x_{t+1}-x_t}} } \nonumber \\
            & \quad + L_2 \pa{ \pa{ \E[t]{ \norm[\nu+1]{x_{t+1}-x_t} } }^{\frac{\nu}{\nu+1}} + \E[t]{\norm[\nu]{x_{t+1}-x_t}} }, \label{eq:model_based_descent_lem_4}
        \end{align}
    \end{subequations}
    where \cref{eq:model_based_descent_lem_1} follows from \cref{eq:stochastic_one_sided_model_2}, \cref{eq:model_based_descent_lem_2} follows from \cref{eq:stochastic_one_sided_model_3}, \cref{eq:model_based_descent_lem_3} follows from \cref{eq:model_based_lipschitz}, and finally \cref{eq:model_based_descent_lem_4} uses Hölder's inequality. Then, the definition for $z_t$ gives:
    \begin{align*}
        & F(z_t) + \frac{\brho}{p}\norm[p]{z_t-x_t} + \frac{\brho}{2}\norm[2]{z_t-x_t} \\
        & \leq F(x_{t+1}) + \frac{\brho}{p}\norm[p]{x_{t+1}-x_t} + \frac{\brho}{2}\norm[2]{x_{t+1}-x_t}.
    \end{align*}

    Letting $\delta_t = \pa{ \E[t]{ \norm[2]{x_{t+1}-x_t} } }^{1/2}$, by Jensen's inequality, one has:
    \begin{align*}
        & \E[t]{\norm[]{x_{t+1}-x_t}} \leq \delta_t, \quad 
        \E[t]{\norm[2]{x_{t+1}-x_t}} = \delta_t^2, \\
        & \E[t]{ \norm[p]{x_{t+1}-x_t} } \leq \delta_t^p, \quad \E[t]{ \norm[\nu]{x_{t+1}-x_t} } \leq \delta_t^{\nu},\\
        & \pa{ \E[t]{\norm[\nu+1]{x_{t+1}-x_t} } }^{\frac{\nu}{\nu+1}} \leq \delta_t^{\nu}.
    \end{align*}

    Plugging back into \cref{eq:model_based_descent_lem_4}, after rearranging and dividing two sides by $(\beta_t-\eta)/2$, one arrives at:
    \begin{align}\label{eq:model_based_descent_lem_6}
        & \E[t]{\norm[2]{z_t-x_{t+1}}} \nonumber \\
        & \leq \norm[2]{z_t-x_t} - \frac{A_{\brho, \eta}}{\beta_t - \eta} \norm[2]{z_t-x_t} - \frac{B_{\brho, \tau, p}}{\beta_t-\eta} \norm[p]{z_t-x_t} + \mathcal{P}(\delta_t)(\beta_t-\eta)^{-1},
    \end{align}
    with
    \begin{align}
       & A_{\brho, \eta} = \brho - \eta, \quad B_{\brho, \tau, p} = \frac{2(\brho-\tau)}{p}, \label{eq:model_based_constant_1} \\
       & \mathcal{P}(\delta_t) = - (\beta_t-\brho) \delta_t^2 + \frac{2(\brho+\tau)}{p} \delta_t^p + 4L_1\delta_t + 4L_2 \delta_t^{\nu} \nonumber.
    \end{align}

    Now we will try to derive an upper bound for $\mathcal{P}(\delta_t)$ that does not depend on $\delta_t$. Thanks to Young's inequality:
    \begin{align*}
        ab \leq \frac{a^p}{p} + \frac{b^q}{q}, \quad a, b > 0, ~ p, q > 1, ~ \frac{1}{p} + \frac{1}{q} = 1,
    \end{align*}
    one can have, if $p \in (1, 2)$,
    \begin{align*}
        \frac{2(\brho+\tau)}{p} \delta_t^p
        &\leq \frac{\beta_t-\brho}{3} \delta_t^2 + (2-p) p^{-1} 3^{\frac{p}{2-p}} (\brho+\tau)^{\frac{2}{2-p}} (\beta_t-\brho)^{-\frac{p}{2-p}}, \\
        4L_1 \delta_t &\leq \frac{\beta_t-\brho}{3}\delta_t^2 + 12L_1^2 (\beta_t-\brho)^{-1} \\
        4L_2 \delta_t^{\nu} & \leq \frac{\beta_t-\brho}{3}\delta_t^2 + (2-\nu) 2^{\frac{2}{2-\nu}} (3\nu)^{\frac{\nu}{2-\nu}} L_2^{\frac{2}{2-\nu}} (\beta_t-\brho)^{-\frac{\nu}{2-\nu}}.
    \end{align*}

    If $p=2$,
    \begin{align*}
        4L_1 \delta_t &\leq \frac{\beta_t-2\brho-\tau}{2} \delta_t^2 + 8 L_1^2 (\beta_t-2\brho-\tau)^{-1} \\
        4L_2 \delta_t^{\nu} &\leq \frac{\beta_t-2\brho-\tau}{2} \delta_t^2 + (2-\nu) 2^{\frac{2}{2-\nu}} (2\nu)^{\frac{\nu}{2-\nu}} (\beta_t-2\brho-\tau)^{-\frac{\nu}{2-\nu}}.
    \end{align*}

    Define the following constants:
    \begin{align}
        C_{\brho, \tau, p} &= \begin{cases}
            (2-p) p^{-1} 3^{\frac{p}{2-p}} (\brho+\tau)^{\frac{2}{2-p}} & \text{if} ~ p \in (1, 2), \\
            0 & \text{if} ~ p = 2,
        \end{cases} \label{eq:model_based_constant_2} \\
        D &= \begin{cases}
            12 & \text{if} ~ p \in (1, 2), \\
            8 & \text{if} ~ p = 2,
        \end{cases} \label{eq:model_based_constant_3} \\
        E_{\nu} &= \begin{cases}
            (2-\nu) 2^{\frac{2}{2-\nu}} (3\nu)^{\frac{\nu}{2-\nu}} & \text{if} ~ p \in (1, 2), \\
            (2-\nu) 2^{\frac{2}{2-\nu}} (2\nu)^{\frac{\nu}{2-\nu}} & \text{if} ~ p = 2,
        \end{cases} \label{eq:model_based_constant_4}
    \end{align}
    then $\mathcal{P}(\delta_t)$ can be upper bounded as:
    \begin{align}\label{eq:model_based_descent_lem_5}
        \mathcal{P}(\delta_t) \leq C_{\brho, \tau, p} (\beta_t-\trho)^{-\frac{p}{2-p}} + D L_1^2 (\beta_t-\trho)^{-1} + E_{\nu} L_2^{\frac{2}{2-\nu}} (\beta_t -\trho)^{-\frac{\nu}{2-\nu}},
    \end{align}
    which, together with \cref{eq:model_based_descent_lem_6} proves the lemma.
\end{proof}

\begin{lemma}\label{lem:model_based_lem_p}
    Under the setting of \cref{lem:model_based_lem_1}, the following holds:
    \begin{align*}
        & \E[t]{ \norm[p]{z_t-x_{t+1}} } \\
        & \leq \norm[p]{z_t-x_t} - \frac{p}{2} (\beta_t-\eta)^{-1} \pa{ A_{\brho, \eta} \norm[p]{z_t-x_t} + B_{\brho, \tau, p} \norm[2(p-1)]{z_t-x_t} } \nonumber \\
    & \quad + (\beta_t-\eta)^{-\frac{p}{2}} \left( (C_{\brho, \tau, p})^{\frac{p}{2}} (\beta_t-\trho)^{-\frac{p^2}{2(2-p)}} + D^{\frac{p}{2}} L_1^p (\beta_t-\trho)^{-p/2} \right. \nonumber \\
        & \qquad\qquad\qquad\qquad \left. + (E_{\nu})^{\frac{p}{2}} L_2^{\frac{p}{2-\nu}} (\beta_t-\trho)^{-\frac{p\nu}{2(2-\nu)}} \right) \\
    & \quad + \frac{p}{2} (\beta_t-\eta)^{-1} G_p \left( A_{\brho, \eta} (B_{\brho, \tau, p})^{\frac{p}{2-p}} (\beta_t-\trho)^{-\frac{p}{2-p}} + (B_{\brho, \tau, p})^{\frac{p}{2-p}} (\beta_t-\trho)^{-\frac{2(p-1)}{2-p}} \right),
    \end{align*}
    with
    \begin{align}
        G_p = \begin{cases}
            1 & \text{if} ~ p \in (1, 2),\\
            0 & \text{if} ~ p = 2.
        \end{cases} \label{eq:model_based_constant_5}
    \end{align}
\end{lemma}

\begin{proof}
    To simply the notation, let $\mathcal{P}$ be the upper bound of $\mathcal{P}(\delta_t)$ in \cref{eq:model_based_descent_lem_5}. Then, by Jensen's inequality, one has:
    \begin{align}
        & \E[t]{ \norm[p]{z_t-x_{t+1}} }
        \leq \pa{ \E[t]{\norm[2]{z_t-x_{t+1}}} }^{p/2} \nonumber \\
        & \leq \left( \norm[2]{z_t-x_{t}} - (\beta-\eta)^{-1} \left( A_{\brho, \eta} \norm[2]{z_t-x_t} + B_{\brho, \tau, p} \norm[p]{z_t-x_t} \right) \right. \nonumber \\
        & \qquad + \left. (\beta_t-\eta)^{-1} \mathcal{P} \right)^{p/2}, \label{eq:model_based_lem_p}
    \end{align}
    where the last inequality follows from the estimate in \cref{lem:model_based_lem_1}. We distinguish the following two cases:

    \textbf{Case 1.} $ \norm[2]{z_t-x_{t}} - (\beta-\eta)^{-1} \left( A_{\brho, \eta} \norm[2]{z_t-x_t} + B_{\brho, \tau, p} \norm[p]{z_t-x_t} \right) > 0 $. Applying successively the inequalities in \cref{useful-ineq}, one arrives at:
    \begin{subequations}
        \begin{align*}
            &\E[t]{ \norm[p]{z_t-x_{t+1}} } \\
            & \leq \left( \norm[2]{z_t-x_t} - (\beta_t-\eta)^{-1} \left( A_{\brho, \eta} \norm[2]{z_t-x_t} + B_{\brho, \tau, p} \norm[p]{z_t-x_t} \right) \right)^{p/2} \\
            &\quad + (\beta_t-\eta)^{-\frac{p}{2}} \mathcal{P}^{p/2} \\
            & \leq \norm[p]{z_t-x_t} - \frac{p}{2} (\beta_t-\eta)^{-1} \pa{ A_{\brho, \eta} \norm[p]{z_t-x_t} + B_{\brho, \tau, p} \norm[2(p-1)]{z_t-x_t} } \nonumber \\
            & \quad + (\beta_t-\eta)^{-\frac{p}{2}} \left( (C_{\brho, \tau, p})^{\frac{p}{2}} (\beta_t-\trho)^{-\frac{p^2}{2(2-p)}} + D^{\frac{p}{2}} L_1^p (\beta_t-\trho)^{-p/2} \right. \nonumber \\
                & \qquad\qquad\qquad\qquad \left. + (E_{\nu})^{\frac{p}{2}} L_2^{\frac{p}{2-\nu}} (\beta_t-\trho)^{-\frac{p\nu}{2(2-\nu)}} \right)
        \end{align*}
    \end{subequations}

    \textbf{Case 2.} $\norm[2]{z_t-x_{t}} - (\beta-\eta)^{-1} \left( A_{\brho, \eta} \norm[2]{z_t-x_t} + B_{\brho, \tau, p} \norm[p]{z_t-x_t} \right) \leq 0$. This condition, together with the definition of $A_{\brho, \eta}$ in \cref{eq:model_based_constant_1} give:
    \begin{align}\label{eq:model_based_lem_p_1}
        \norm[2-p]{z_t-x_t} \leq \frac{B_{\brho, \tau, p} (\beta_t-\eta)^{-1}}{1-(\beta_t-\eta_t)^{-1} A_{\brho, \eta}} = B_{\brho, \tau, p} (\beta_t-\brho)^{-1} \leq B_{\brho, \tau, p} (\beta_t-\trho)^{-1}
    \end{align}

    Note that, in the case where $p=2$, this case can not happen, since the condition is equivalent to
    \begin{align*}
        \frac{2\brho-\tau-\eta}{\beta_t-\eta} \geq 1,
    \end{align*}
    which contradicts the condition $\beta_t > \trho = 2\brho+\tau$. Therefore, by using the fact that $a\mapsto a^{p/2}$ is non-increasing for $a\geq 0$ and $p \in (1, 2]$,
    from \cref{eq:model_based_lem_p}, one obtains:
    \begin{subequations}
        \begin{align*}
            &\E[t]{\norm[p]{z_t-x_t}}
            \leq \mathcal{P}^{p/2} \\
            & \leq \norm[p]{z_t-x_t} - \frac{p}{2} (\beta_t-\eta)^{-1} \pa{ A_{\brho, \eta} \norm[p]{z_t-x_t} + B_{\brho, \tau, p} \norm[2(p-1)]{z_t-x_t} } \\
            & \quad + \mathcal{P}^{p/2}
            + \frac{p}{2} (\beta_t-\eta)^{-1} \pa{ A_{\brho, \eta} \norm[p]{z_t-x_t} + B_{\brho, \tau, p} \norm[2(p-1)]{z_t-x_t} } \\
            & \leq \norm[p]{z_t-x_t} - \frac{p}{2} (\beta_t-\eta)^{-1} \pa{ A_{\brho, \eta} \norm[p]{z_t-x_t} + B_{\brho, \tau, p} \norm[2(p-1)]{z_t-x_t} } \nonumber \\
            & \quad + (\beta_t-\eta)^{-\frac{p}{2}} \left( (C_{\brho, \tau, p})^{\frac{p}{2}} (\beta_t-\trho)^{-\frac{p^2}{2(2-p)}} + D^{\frac{p}{2}} L_1^p (\beta_t-\trho)^{-p/2} \right. \nonumber \\
                & \qquad\qquad\qquad\qquad \left. + (E_{\nu})^{\frac{p}{2}} L_2^{\frac{p}{2-\nu}} (\beta_t-\trho)^{-\frac{p\nu}{2(2-\nu)}} \right) \\
            & \quad + \frac{p}{2} (\beta_t-\eta)^{-1} G_p \left( A_{\brho, \eta} (B_{\brho, \tau, p})^{\frac{p}{2-p}} (\beta_t-\trho)^{-\frac{p}{2-p}} + (B_{\brho, \tau, p})^{\frac{p}{2-p}} (\beta_t-\trho)^{-\frac{2(p-1)}{2-p}} \right),
        \end{align*}
    \end{subequations}
    where the last inequality follows from the estimate \cref{eq:model_based_lem_p_1} and uses
    \begin{align*}
        G_p = \begin{cases}
            1 & \text{if} ~ p \in (1, 2),\\
            0 & \text{if} ~ p = 2.
        \end{cases}
    \end{align*} 
    
    Combining the results from the two cases proves the lemma.
\end{proof}

We are now ready to prove the convergence of \cref{alg:model_based}.

\begin{theorem}\label{thm:model_based}
    Under \cref{assump:stochastic_one_sided_model}, fix $\brho > 2(2\tau+\eta)$ and a sequence $\set{\beta_t} \in (\trho, +\infty)$, with $\trho$ defined in \cref{eq:model_based_constant_6}, the sequence $\set{x_t}$ generated by \cref{alg:model_based} satisfies:
    \begin{align}\label{eq:model_based_conv_1}
    &\E{ \sum_{t=0}^T \mathcal{A} (\beta_t-\eta)^{-1} d \left(0, \moreaup[1/\brho]{x_t} \right)^2 } \nonumber \nonumber \\
    &\le 2 \brho \pa{ \moreaup[1/\brho]{x_0} - \min F } \nonumber \\
    &\quad+ \brho^2 \mathcal{B} \left( (\beta_t-\trho)^{-\frac{2}{2-p}}
    + L_1^2 (\beta_t-\trho)^{-2}
    + (\beta_t-\trho)^{-\frac{p}{2-p}}
    + L_1^p (\beta_t-\trho)^{-p} \right. \nonumber \\ 
    &\qquad\qquad \left.+ L_2^{\frac{2}{2-\nu}} (\beta_t-\trho)^{-\frac{2}{2-\nu}}
    + L_2^{\frac{p}{2-\nu}} (\beta_t-\trho)^{-\frac{p}{2-\nu}} \right),
    \end{align}
    with constants
    \begin{align*}
        &\mathcal{A} = \min \set{ A_{\brho, \eta}, B_{\brho, \tau, p} }, \\
        &\mathcal{B} = \max \left\{ G_p A_{\brho, \eta} (B_{\brho, \tau, p})^{\frac{p}{2-p}} + C_{\brho, \tau, p}, G_p (B_{\brho, \tau, p})^{\frac{p}{2-p}} + \frac{2}{p} (C_{\brho, \tau, p})^{p/2}, \right. \\ 
        &\qquad\qquad~ \left. 2D, E_{\nu}, \frac{2}{p} (E_{\nu})^{p/2} \right\},
    \end{align*}
    $A_{\brho, \eta}$ and $B_{\brho, \tau, p}$ defined in \cref{eq:model_based_constant_1}, $ C_{\brho, \tau, p}$, $D$ and $ E_{\nu} $ defined in \cref{eq:model_based_constant_2,eq:model_based_constant_3,eq:model_based_constant_4}, $G_p$ defined in \cref{eq:model_based_constant_5}.
\end{theorem}

\begin{proof}
    For each $t$, by \cref{thm:generalized_moreau_paraweak}~(iii), one can pick $z_t \in \gproxp[1/\brho]{x_t}$ such that
    \begin{align*}
        d \pa{ 0, \moreaup[1/\brho]{x_t} } = \brho\pa{ \norm[2]{z_t-x_t} + \norm[p-1]{z_t-x_t} }.
    \end{align*}
    
    From the definition of the generalized Moreau envelope, and the estimates in \cref{lem:model_based_lem_1,lem:model_based_lem_p}, one successively deduces
    \begin{subequations}
        \begin{align*}
            &\E[t]{ \moreaup[1/\brho]{x_{t+1}} } \\
            &\leq F(z_t) + \E[t]{ \frac{\brho}{2}\norm[2]{z_t-x_{t+1}} + \frac{\brho}{p}\norm[p]{z_t-x_{t+1}} } \\
            &\leq \moreaup[1/\brho]{x_t} - \frac{\brho}{2} (\beta_t-\eta)^{-1} \mathcal{A} \pa{ \norm[2]{z_t-x_t} + 2 \norm[p]{z_t-x_t} + \norm[2(p-1)]{z_t-x_t} } \\
            &\quad + \frac{\brho}{2} \mathcal{B} \left( (\beta_t-\trho)^{-\frac{2}{2-p}}
            + L_1^2 (\beta_t-\trho)^{-2}
            + (\beta_t-\trho)^{-\frac{p}{2-p}}
            + L_1^p (\beta_t-\trho)^{-p} \right. \nonumber \\ 
            &\qquad\qquad \left.+ L_2^{\frac{2}{2-\nu}} (\beta_t-\trho)^{-\frac{2}{2-\nu}}
            + L_2^{\frac{p}{2-\nu}} (\beta_t-\trho)^{-\frac{p}{2-\nu}} \right),
        \end{align*}
    \end{subequations}
    where the last step uses the fact that $\beta_t-\eta \geq \beta_t-\trho$, together with the definition of $\mathcal{A}$ and $\mathcal{B}$. Note that from the definition of $z_t$, one has:
    \begin{align*}
        d \pa{ 0, \moreaup[1/\brho]{x_t} }^2 = \brho^2 \pa{ \norm[2]{z_t-x_t} + 2 \norm[p]{z_t-x_t} + \norm[2(p-1)]{z_t-x_t}}.
    \end{align*}

    Therefore, after rearranging, taking the full expectation, and taking the sum from $t=0$ to $T$, one obtains \cref{eq:model_based_conv_1}.
\end{proof}

In \cref{thm:model_based}, one can notice that the last two terms in \cref{eq:model_based_conv_1} depend on $\beta_t$ with exponents depending on $\nu$. These two terms result from the second term on the RHS of \cref{eq:stochastic_one_sided_model_3}. We highlight that \cref{assump:stochastic_one_sided_model}~(iv) can be seen as a combined Lipschitz-continuous and Hölder-continuous property of the stochastic model $f_x(., \xi)$, which is more general than the Lipschitz-continuous model studied in \cite{davis_stochastic_2019} (see their Assumption B4). As a consequence, when $L_2 \neq 0$, the convergence rate of \cref{alg:model_based} can be slower than the rate $\mathcal{O}\pa{ T^{1/p-1} }$ stated in \cref{thm:conv_psgd} of the stochastic proximal subgradient methods in \cref{alg:psgd} (and therefore the rate $\mathcal{O}(T^{-1/2})$ of the stochastic model-based methods shown in \cite{davis_stochastic_2019} for the case $p=2$); while when $L_2=0$, one can recover the rate $\mathcal{O}\pa{ T^{1/p-1} }$ using the same step size rule in \cref{thm:conv_psgd}. The following two corollaries show the corresponding convergence rates for these two cases.

\begin{corollary}[Convergence rate of \cref{alg:model_based} for Lipschitz-continuous models]\label{cor:model_based_conv_1}
    Under the setting of \cref{thm:model_based}, if $L_2=0$, and one chooses $\beta_t = \trho + (t+1)^{1/p}$, for $\Delta \geq \moreaup[1/\brho]{x_0}-\min F$, then in the case $p \in (1, 2)$, the output $x_{t^*}$ of \cref{alg:model_based} satisfies:
    \begin{align*}
        & \E{ d \pa{ 0, \moreaup[1/\brho]{x_{t^*}} }^2 } \nonumber \\
        & \leq \frac{\brho(\brho-\eta) \pa{ 2 \Delta + \brho \mathcal{B} }}{\mathcal{A}(T+1)} \nonumber \\
        & \quad + \frac{\brho^2 \mathcal{B}}{\mathcal{A} (T+1)^{1-1/p} } \left( 
            \frac{2 \Delta}{\brho \mathcal{B}} + \frac{2}{(p-1)^2+1} + \frac{2L_1^2}{2-p} + \frac{1}{p-1} + L_1^p \pa{ \ln(T+1)+1 }
        \right).
    \end{align*}

    In the case where $p=2$, the output $x_{t^*}$ satisfies:
    \begin{align*}
        & \E{ d \pa{ 0, \moreaup[1/\brho]{x_{t^*}} }^2 } \nonumber \\
        & \leq \frac{\brho(2\brho+\tau-\eta) \pa{ 2 \Delta + \brho \mathcal{B} }}{\mathcal{A}(T+1)}
        + \frac{2 \brho^2 \mathcal{B}}{\mathcal{A} (T+1)^{1/2} } \left( 
            \frac{\Delta}{\brho \mathcal{B}} + 1 + L_1^2 \pa{ \ln(T+1)+1 }
        \right).
    \end{align*}
    
\end{corollary}

\begin{proof}
    Dividing two sides of \cref{eq:model_based_conv_1} by $\sum_{t=0}^T \mathcal{A} (\beta_t-\eta)^{-1}$ gives:
    \begin{align}\label{eq:model_based_cor_1}
        &\E{ d \left(0, \moreaup[1/\brho]{x_{t^*}} \right)^2 } \nonumber \\
        &\le \frac{2 \brho \pa{ \moreaup[1/\brho]{x_0} - \min F }}{\sum_{t=0}^T \mathcal{A}(\beta_t-\eta)^{-1}} \nonumber \\
        &\quad+ \frac{\brho^2 \mathcal{B}}{\sum_{t=0}^T \mathcal{A}(\beta_t-\eta)^{-1}} \sum_{t=0}^T \left( 
            (\beta_t-\trho)^{-\frac{2}{2-p}}
            + L_1^2 (\beta_t-\trho)^{-2}
            + (\beta_t-\trho)^{-\frac{p}{2-p}} \right. \nonumber \\
            &\qquad\qquad\qquad\qquad\qquad\qquad \left.+ L_1^p (\beta_t-\trho)^{-p}
        \right).
    \end{align}

    Then, by appealing to \cref{fact:sum_step_size}, one has the following estimates:
    \begin{align*}
        & \sum_{t=0}^T (\beta_t-\eta)^{-1} \geq \frac{T+1}{\trho-\eta + (T+1)^{1/p}} \\
        & \sum_{t=0}^T (\beta_t-\trho)^{-\frac{2}{2-p}} = \sum_{t=0}^T (t+1)^{-\frac{2}{p(2-p)}}
        \leq \frac{2}{(p-1)^2+1}, \\
        & \sum_{t=0}^T (\beta_t-\trho)^{-2} = \sum_{t=0}^T (t+1)^{-2/p} 
        \leq \begin{cases}
             \frac{2}{2-p} & \text{if}~ p \in (1, 2), \\
             \ \ln(T+1)+1 & \text{if}~ p = 2,
        \end{cases} \\
        & \sum_{t=0}^T (\beta_t-\trho)^{-\frac{p}{2-p}} = \sum_{t=0}^T (t+1)^{-\frac{1}{2-p}}
        \leq \frac{1}{p-1}, \\
        & \sum_{t=0}^T (\beta_t-\trho)^{-p} = \sum_{t=0}^T (t+1)^{-1} \leq \ln(T+1)+1.
    \end{align*}

    Plugging back into \cref{eq:model_based_cor_1} proves the results.
\end{proof}

The results in \cref{cor:model_based_conv_1} for Lipschitz-continuous models recover the rate $\mathcal{O}(T^{1/p-1})$ of \cref{alg:psgd}, and hence the rate $\mathcal{O}(T^{-1/2})$ proven in \cite{davis_stochastic_2019}. For models that are not necessarily Lipschitz, i.e., $L_2\neq 0$, notice that the smallest exponent appearing in \cref{eq:model_based_conv_1} is $p/(2-\nu)$, which can be smaller than or equal to $1$ for $p+\nu\leq2$, and hence in this case the term with this exponent might not vanish as $T \to \infty$ regardless of the choice of step sizes. Therefore, in order to establish the convergence of \cref{alg:model_based} for Lipschitz and Hölder-continuous models, we might need to impose an additional assumption $p+\nu > 2$. The following corollary gives the convergence rate in this setting.

\begin{corollary}[Convergence of \cref{alg:model_based} for Lipschitz and \\ Hölder-continuous models]\label{cor:model_based_conv_2}
    Under the setting of \cref{thm:model_based}, if one chooses $\beta_t = \trho + (t+1)^{(2-\nu)/p}$, for $\Delta \geq \moreaup[1/\brho]{x_0}-\min F$, then in the case $p \in (1, 2),~ \nu \in (2-p, 1)$, the output $x_{t^*}$ of \cref{alg:model_based} satisfies:
    \begin{align*}
        & \E{ d \pa{ 0, \moreaup[1/\brho]{x_{t^*}} }^2 } \nonumber \\
        & \leq \frac{\brho(\brho-\eta) \pa{ 2 \Delta + \brho \mathcal{B} }}{\mathcal{A}(T+1)} \nonumber \\
        & \quad + \frac{\brho^2 \mathcal{B}}{\mathcal{A} (T+1)^{1-\frac{2-\nu}{p}}} \left( 
            \frac{2 \Delta}{\brho \mathcal{B}} + \frac{2(2-\nu)}{(p-1)^2+3-2\nu} + \frac{2(2-\nu)}{2(2-\nu)-p} L_1^2 + \frac{2-\nu}{p-\nu} \right. \\
            &\qquad\qquad\qquad\qquad\qquad \left. + \frac{2-\nu}{1-\nu} L_1^p + \frac{2}{2-p} L_2^{\frac{2}{2-\nu}}
            + L_2^{\frac{p}{2-\nu}} \pa{ \ln(T+1)+1 }
        \right).
    \end{align*}
    In the case where $p \in (1, 2),~ \nu=1$, the output $x_{t^*}$ satisfies:
    \begin{align*}
        \E{ d \pa{ 0, \moreaup[1/\brho]{x_{t^*}} }^2 }
        & \leq \frac{\brho(\brho-\eta) \pa{ 2 \Delta + \brho \mathcal{B} }}{\mathcal{A}(T+1)} \\
        & \quad + \frac{\brho^2 \mathcal{B}}{\mathcal{A} (T+1)^{1-1/p} } \left( 
            \frac{2 \Delta}{\brho \mathcal{B}} + \frac{2}{(p-1)^2+1} + \frac{2(L_1^2+L_2^2)}{2-p} \right. \\
        &\qquad\qquad\qquad\qquad+ \left. 
        \frac{1}{p-1} + (L_1^p+L_2^p) \pa{ \ln(T+1)+1 }
        \right).
    \end{align*}
    In the case where $p=2,~ \nu \in (0, 1)$, the output $x_{t^*}$ satisfies:
    \begin{align*}
        & \E{ d \pa{ 0, \moreaup[1/\brho]{x_{t^*}} }^2 } \nonumber \\
        & \leq \frac{\brho(2\brho+\tau-\eta) \pa{ 2 \Delta + \brho \mathcal{B} }}{\mathcal{A}(T+1)} \\
        &\quad+ \frac{2 \brho^2 \mathcal{B}}{\mathcal{A} (T+1)^{\nu/2} } \left( 
            \frac{\Delta}{\brho \mathcal{B}} + 1 + \frac{2-\nu}{1-\nu} L_1^2 + \pa{ \ln(T+1)+1 } L_2^{\frac{2}{2-\nu}}
        \right).
    \end{align*}
    In the case where $p=2,~ \nu =1$, the output $x_{t^*}$ satisfies:
    \begin{align*}
        & \E{ d \pa{ 0, \moreaup[1/\brho]{x_{t^*}} }^2 } \nonumber \\
        & \leq \frac{\brho(2\brho+\tau-\eta) \pa{ 2 \Delta + \brho \mathcal{B} }}{\mathcal{A}(T+1)}
        + \frac{2 \brho^2 \mathcal{B}}{\mathcal{A} (T+1)^{1/2} } \left( 
            \frac{\Delta}{\brho \mathcal{B}} + 1 + \pa{ \ln(T+1)+1 } (L_1^2+L_2^2)
        \right).
    \end{align*}
\end{corollary}

\begin{proof}
    The proof follows exactly the same steps as in the proof of \cref{cor:model_based_conv_1}, with the following estimates:
    \begin{align*}
        & \sum_{t=0}^T (\beta_t-\eta)^{-1} \geq \frac{T+1}{\brho-\eta + (T+1)^{\frac{2-\nu}{p}}} \\
        & \sum_{t=0}^T (\beta_t-\trho)^{-\frac{2}{2-p}} \leq \frac{2(2-\nu)}{(p-1)^2+3-2\nu}, \\
        & \sum_{t=0}^T (\beta_t-\trho)^{-2}
        \leq \begin{cases}
             \frac{2(2-\nu)}{2(2-\nu)-p} & \text{if}~ p \in (1, 2),~ \nu \in (2-p, 1) \\
             \frac{2}{2-p} & \text{if}~ p \in (1, 2),~ \nu=1,\\
             \frac{2-\nu}{1-\nu} & \text{if} ~ p = 2,~ \nu \in (0, 1), \\
             \ \ln(T+1)+1 & \text{if}~ p = 2,~ \nu=1,
        \end{cases} \\
        & \sum_{t=0}^T (\beta_t-\trho)^{-\frac{p}{2-p}} \leq \frac{2-\nu}{p-\nu}, \\
        & \sum_{t=0}^T (\beta_t-\trho)^{-p}
        \leq \begin{cases}
             \frac{2-\nu}{1-\nu} & \text{if}~ \nu \in (2-p, 1), \\
             \ \ln(T+1)+1 & \text{if}~ \nu=1,
        \end{cases} \\
        & \sum_{t=0}^T (\beta_t-\trho)^{-\frac{2}{2-\nu}}
        \leq \begin{cases}
             \frac{2}{2-p} & \text{if}~ p \in (1, 2), \\
             \ \ln(T+1)+1 & \text{if}~ p = 2,
        \end{cases} \\
        & \sum_{t=0}^T (\beta_t-\trho)^{-\frac{p}{2-\nu}}
        \leq \ln(T+1)+1.
    \end{align*}
\end{proof}

\subsection{Some algorithmic examples} 

We next look at some specific examples of \cref{alg:model_based}. For each case, we list the standard assumption under which the method is applicable and satisfies \cref{assump:stochastic_one_sided_model} for some $\tau, \eta, L > 0, ~ \nu \in (0, 1], p \in (1, 2)$. Convergence guarantees for each method then follow immediately from \cref{cor:model_based_conv_1,cor:model_based_conv_2}.

\textbf{Stochastic proximal subgradient} Consider the problem \cref{eq:cp_prob_model_based}, together with the following assumptions:
\begin{itemize}
    \item[$(\text{A}1)$] It is possible to draw i.i.d. realizations $\xi_1, \xi_2, \dots$ from $\pi$.
    \item[$(\text{A}2)$] The function $f$ is $(\rho_1, 0, p)$-paraconvex and the function $\varphi$ is $\rho_2$-weakly convex for some $\rho_1, \rho_2 \geq 0$ and $p \in (1, 2]$.
    \item[$(\text{A}3)$] There exists an open set $U \supset \dom{\varphi}$ and a measurable mapping $G: U \times \Omega \to \R$ such that $ \E[\xi \sim \pi]{G(x, \xi)} \in \partial f(x) $ for all $x \in U$.
    \item[$(\text{A}4)$] There exists a constant $L \geq 0$ such that $\E[\xi \sim \pi]{ \norm[2]{G(x, \xi)} } \leq L^2$ for all $x \in U$. 
\end{itemize}

The stochastic proximal subgradient method is a special case of \cref{alg:model_based} with the linear models
\begin{align*}
    f_x(y, \xi) = f(x, \xi) + \langle G(x, \xi), y-x \rangle.
\end{align*}

Obviously, \cref{assump:stochastic_one_sided_model}~(i) and \cref{assump:stochastic_one_sided_model}~(iii) are immediate with $\eta=\rho_2$, \cref{assump:stochastic_one_sided_model}~(ii) with $\tau=\rho_1$ follows from $(\text{A}3)$ and the subdifferential characterization of paraconvex function in \cref{charac-paraweak-ineq1}. \cref{assump:stochastic_one_sided_model}~(iv) with $L_1=L, L_2=0, \nu=0$ also follows immediately from $(\text{A}4)$.

We highlight that the above set of assumptions does not cover the setting studied in \cref{sec:proximal} for the stochastic proximal subgradient method. In particular, (A4) requires that the stochastic subgradients have bounded second moment, whereas \cref{assump:stochastic_grad_oracle}~(iii) only requires a bounded $p$-th moment.

\textbf{Stochastic prox-linear} Consider the problem \cref{eq:cp_prob_model_based} with
\begin{align*}
    f(x) = \E[\xi\sim\pi]{h(c(x, \xi), \xi)},
\end{align*}
and the following set of assumptions, which hold on an open set $U \supset \dom{\varphi}$:
\begin{itemize}
    \item[$(\text{B}1)$] We can draw i.i.d. realizations $\xi_1, \xi_2, \dots$ from $\pi$.
    \item[$(\text{B}2)$] The mappings $h: \R^m \times \Omega \to \R$ and $c: U \times \Omega \to \R^m$ are measurable.
    \item[$(\text{B}3)$] The function $\varphi$ is $\rho_1$-weakly convex.
    \item[$(\text{B}4)$] There exists square integrable functions $\ell_1, \ell_2, \ell_3: \Omega \to \R_+$ such that for $\pi$-a.e. $\xi \in \Omega$, the function $z \mapsto h(z, \xi)$ is $\rho_2$-weakly convex and $\ell_1(\xi)$-Lipschitz, the map $x \mapsto c(x, \xi)$ is $C^1$-smooth with $(\nu, \ell_2(\xi))$-Hölder continuous Jacobian with $\nu \in (0, 1]$, and $\left\lVert \nabla c(x, \xi) \right\rVert_{\text{op}} \leq \ell_3(\xi)$ for all $x \in U$.
\end{itemize}

The stochastic prox-linear method is \cref{alg:model_based} with the following weakly convex model:
\begin{align*}
    f_x(y, \xi) = h(c(x, \xi) + \nabla c(x, \xi)(y-x), \xi).
\end{align*}

Note that \cref{assump:stochastic_one_sided_model}~(iii) is immediate with $\eta=\rho_1+\rho_2$. \cref{assump:stochastic_one_sided_model}~(ii) holds with $\tau=\pa{\E[\xi]{ \ell_1(\xi)^2 }}^{1/2} \pa{\E[\xi]{\ell_2(\xi)^2}}^{1/2}$ and $p=\nu+1$:
\begin{align*}
    \E[\xi]{ f_x(y, \xi) - f(y) }
    &= \E[\xi]{ h(c(x, \xi) + \nabla c(x, \xi) (y-x), \xi) - h(c(y, \xi), \xi) } \\
    &\leq \E[\xi]{ \ell_1(\xi)  \norm{ c(x, \xi) - c(y, \xi) + \nabla c(x, \xi)(y-x) } } \\
    &\leq \pa{\E[\xi]{ \ell_1(\xi)^2 }}^{1/2} \pa{\E[\xi]{\ell_2(\xi)^2}}^{1/2} \norm[\nu+1]{x-y},
\end{align*}
where the first inequality is from the Lipschitzness of $h(., \xi)$, and the second inequality follows from the Hölder-continuity of $\nabla c(., \xi)$. $(\text{B}4)$ also implies \cref{assump:stochastic_one_sided_model}~(iv) with $L_1 = \pa{\E[\xi]{ \ell_1(\xi)^2 }}^{1/2} \pa{\E[\xi]{\ell_3(\xi)^2}}^{1/2}$ and $L_2=0$:
\begin{align*}
    f_x(x, \xi) - f_x(y, \xi)
    &= h(c(x, \xi), \xi) - h(c(x, \xi) + \nabla c(x, \xi)(y-x), \xi) \\
    & \leq \ell_1(\xi) \norm{ \nabla c(x, \xi) (y-x) } \\
    & \leq \ell_1(\xi) \ell_3(\xi) \norm{x-y}.
\end{align*}

We can weaken $(\text{B}4)$ by allowing $h(., \xi)$ to be only Hölder-continuous:
\begin{itemize}
    \item[$(\bar{\text{B}}4)$] There exists functions $\ell_1, \ell_2, \ell_3: \Omega \to \R_+$, for $\nu_1 \in (0, 1]$, $\E[\xi]{\ell_1(\xi)^{\nu_1+1}} < +\infty$, $\E[\xi]{\ell_2(\xi)^{2\nu_1}} < +\infty$ and $\E[\xi]{\ell_3(\xi)^{\nu_1(\nu_1+1)}} < +\infty$; for $\pi$-a.e. $\xi \in \Omega$, the function $z \mapsto h(z, \xi)$ is $\rho_2$-weakly convex and $(\nu_1, \ell_1(\xi))$-Hölder continuous, the map $x \mapsto c(x, \xi)$ is $C^1$-smooth with $(\nu_2, \ell_2(\xi))$-Hölder continuous Jacobian with $\nu_2 \in (0, 1], 1 < \nu_1(\nu_2+1) \leq 2$, and $\left\lVert \nabla c(x, \xi) \right\rVert_{\text{op}} \leq \ell_3(\xi)$ for all $x \in U$.
\end{itemize}

In this case, \cref{assump:stochastic_one_sided_model}~(ii) holds with $\tau=\pa{\E[\xi]{ \ell_1(\xi)^2 }}^{1/2} \pa{\E[\xi]{\ell_2(\xi)^{2\nu_1}}}^{1/2}$ and $p=\nu_1(\nu_2+1) \in (1, 2]$:
\begin{align*}
    \E[\xi]{ f_x(y, \xi) - f(y) }
    &= \E[\xi]{ h(c(x, \xi) + \nabla c(x, \xi) (y-x), \xi) - h(c(y, \xi), \xi) } \\
    &\leq \E[\xi]{ \ell_1(\xi)  \norm[\nu_1]{ c(x, \xi) - c(y, \xi) + \nabla c(x, \xi)(y-x) } } \\
    &\leq \pa{\E[\xi]{ \ell_1(\xi)^2 }}^{1/2} \pa{\E[\xi]{\ell_2(\xi)^{2\nu_1}}}^{1/2} \norm[\nu_1(\nu_2+1)]{x-y}.
\end{align*}

\cref{assump:stochastic_one_sided_model}~(iv) with $L_2=\pa{\E[\xi]{ \ell_1(\xi)^{\nu_1+1} }}^{\frac{1}{\nu_1+1}} \pa{\E[\xi]{\ell_3(\xi)^{\nu_1(\nu_1+1)}}}^{\frac{1}{\nu_1+1}}$, $L_1=0$ and $\nu=\nu_1$ also follows from:
\begin{align*}
    f_x(x, \xi) - f_x(y, \xi)
    &= h(c(x, \xi), \xi) - h(c(x, \xi) + \nabla c(x, \xi)(y-x), \xi) \\
    & \leq \ell_1(\xi) \norm[\nu_1]{ \nabla c(x, \xi) (y-x) } \\
    & \leq \ell_1(\xi) \ell_3(\xi)^{\nu_1} \norm[\nu_1]{x-y}.
\end{align*}

\section{Numerical experiments}\label{sec:experiment}

\begin{figure}
    \centering
    
    \begin{subfigure}{.9\linewidth}
        \begin{center}
            \centerline{\includegraphics[width=\columnwidth]{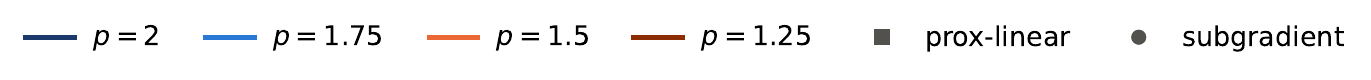}}
        \end{center}
    \end{subfigure}
    
    \begin{subfigure}{.45\linewidth}
        \begin{center}
            \centerline{\includegraphics[width=\columnwidth]{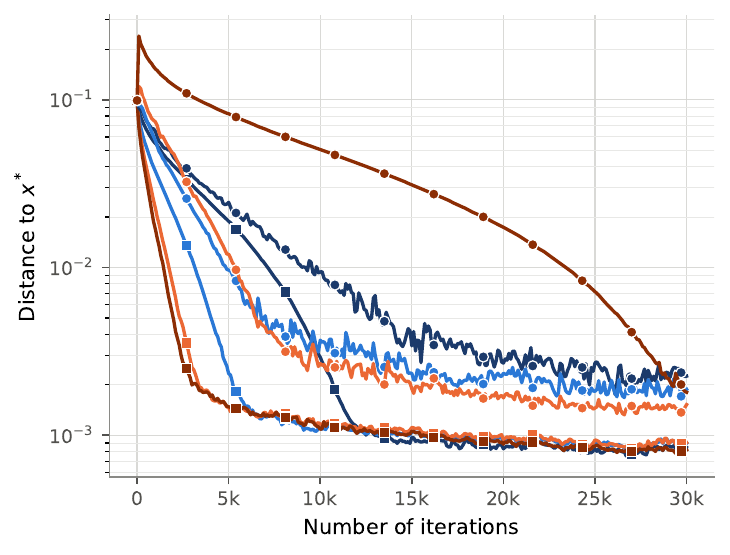}}
        \end{center}
    \end{subfigure}
    \begin{subfigure}{.45\linewidth}
        \begin{center}
            \centerline{\includegraphics[width=\columnwidth]{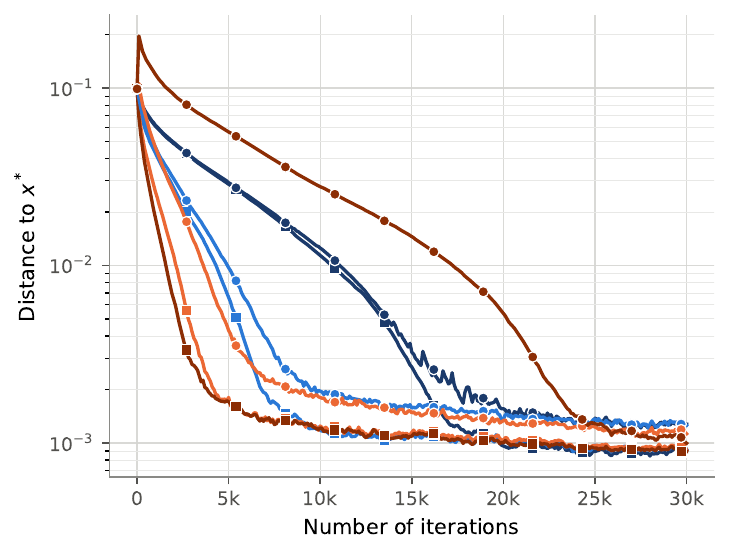}}
        \end{center}
    \end{subfigure}
    
    \caption{Distance to the true signal $x^*$ of the stochastic subgradient and the stochastic prox-linear method with $p \in \set{1.25, 1.5, 1.75, 2}$ and $\kappa = 2.5$ (on the left), $\kappa=3$ (on the right).}
    \label{fig:exp1}
\end{figure}


\begin{table}[t]
\centering
\caption{Speedup ratio $T_2/T_p$, where $T_p$ denotes the number of iterations needed to reach $\mathrm{dist}(x_t, x^*)/\|x^*\|\le 5\cdot10^{-3}$; values above $1$ indicate faster convergence. $^{\dagger}$Not all $20$ instances reached the target distance.}
\label{tab:iters_to_target}
\begin{tabular}{llrrr}
\toprule
Method & $\kappa_1$ & \multicolumn{3}{c}{Speedup ratio $T_2/T_p$} \\
\cmidrule(lr){3-5}
 & & $p=1.75$ & $p=1.5$ & $p=1.25$ \\
\midrule
Subgradient & $2.5$ & 1.9$\times^{\dagger}$ & 1.9$\times^{\dagger}$ & 0.5$\times^{\dagger}$ \\
 & $3$ & 2.1$\times$ & 3.0$\times$ & 0.7$\times^{\dagger}$ \\
\midrule
Prox-linear & $2.5$ & 2.2$\times$ & 3.7$\times$ & 4.3$\times$ \\
 & $3$ & 2.4$\times$ & 4.8$\times$ & 6.1$\times$ \\
\bottomrule
\end{tabular}
\end{table}

    
    

    

In this section, we perform numerical experiments on the robust phase retrieval problem with non-Gaussian channels and heavy-tailed noise. Our main goal is to study the benefit of using a paraconvex loss function instead of the classical weakly convex objective \cite{davis_stochastic_2019} in order to recover the true signal from noisy measurements. The setup is as follows: the true signal $x^*$ uniformly on the unit sphere, and the initial point $x_0$ uniformly inside a ball around $x^*$; the channels are non Gaussian: each entry of $a_i$ is generated independently from a Student-t distribution with degree-of-freedom (dof) $\kappa_1$; the measurements $b_i = \langle a_i, x^* \rangle^2 + \sigma \zeta_i$ is corrupted with noise $\zeta_i$ also drawn from a Student-t distribution with dof $\kappa_2$, and $\sigma$ is the noise level.

We then apply the stochastic subgradient and the stochastic prox-linear method to solve the following problem
\begin{align}\label{eq:exp:obj}
    \min_{x \in \R^n} \frac{1}{m} \sum_{i=1}^m \abs{ \abs{\langle a_i, x \rangle}^p - \operatorname{sign}(b) \abs{b}^{p/2} }
\end{align}
for $p \in \set{1.25, 1.5, 1.75, 2}$. This objective function is $(\rho, p)$-paraconvex as shown in \cref{example:composite_paraconvex}. Since our algorithms only use one data sample at a time, let us define the following functions
\begin{align*}
    h_i(x) &= \abs{ c_i(x) }, \\
    c_i(x) &= \abs{ \langle a, x \rangle }^p - \operatorname{sign}(b) \abs{b}^{p/2}. 
\end{align*}

Note that $c_i(x)$ is differentiable over $\R^n$ with
\begin{align*}
    \nabla c_i(x) = pa \abs{ \langle a, x \rangle }^{p-2} \langle a, x \rangle.
\end{align*}

The stochastic subgradient method only needs to pick an element of the subdifferential
\begin{align*}
    \partial h_i(x) = \nabla c_i(x) \left\{\begin{aligned}
        & \operatorname{sign}\pa{ c_i(x) } & \text{if} ~ c_i(x) \neq 0,\\
        &[-1, 1] & \text{otherwise}
    \end{aligned} \right\}.
\end{align*}

The stochastic prox-linear method requires solving the following subproblem for the update direction $d_t$ at each iteration:
\begin{align*}
    d_t = \margmin[d \in \R^d]{ \abs{ c_i(x) + \langle \nabla c_i(x), d \rangle } + \frac{1}{2\gamma_t} \norm[2]{d} },
\end{align*}
where $\gamma_t$ is the step size. Then the next iterate is computed as $x+d_t$. Following a standard Langragian calculation in \cite{duchi_stochastic_2018,davis_stochastic_2019},
\begin{align*}
    d_t = \operatorname{proj}_{[-1, 1]}\pa{ \frac{-c_i(x)}{ \gamma_t \norm[2]{\nabla c_i(x)} } } \nabla c_i(x).
\end{align*}

Henceforth, let us denote by $\set{x^{\dagger}_{p, t}}$ and $\set{x^{\natural}_{p, t}}$ the iterates generated by the stochastic subgradient and the stochastic prox-linear algorithms when minimizing \cref{eq:exp:obj} with exponent $p$. We report the results in terms of the function values and the distance to $x^*$
\begin{align*}
    \operatorname{dist}(x, x^*) = {\min \set{ \norm[2]{x-x^*}, \norm[2]{x+x^*} }}.
\end{align*}

We perform the experiments with the following set of hyperparameters: $n=100, m=300, \sigma=0.001$, dof $\kappa_1 \in \set{1, 2.5}$ and $\kappa_2=1$, and step size $\gamma_t = \gamma_0 (t+1)^{-1/p}$, where $\gamma_0$ is tuned over a logarithmic grid from $10^{-7}$ to $0.1$. Each run consists of $3\times10^4$ iterations; we repeat each set of experiments over $20$ runs with different random seeds and report the medians.

\Cref{fig:exp1} reports the distance to the true signal $x^*$. From this figure, one can see that the stochastic prox-linear method is faster than the stochastic subgradient method. Note that our theoretical guarantees suggest that these two methods exhibit the same worst-case convergence rate, but in practice, the benefit of the stochastic prox-linear method is understandable since the stochastic model employed in this algorithm provides a tighter two-sided approximation, instead of one-sided as in \cref{assump:stochastic_one_sided_model}~(ii). This empirical advantage of the prox-linear method has also been observed in the weakly-convex setting \cite{davis_stochastic_2019,duchi_stochastic_2018}.

More notably, one can also spot the benefit of using a paraconvex reformulation of the robust phase-retrieval problem \cref{eq:exp:obj} with exponent $p < 2$ in this heavy-tail setting, in terms of recovering the true signal $x^*$. To make the comparison more precise, \cref{tab:iters_to_target} reports the speedup in the number of iterations needed to reach $\mathrm{dist}(x_t, x^*)/\|x^*\|\le 5\cdot10^{-3}$, measured against the weakly convex case $p=2$. For the stochastic prox-linear method, the speedup grows steadily as $p$ decreases. The stochastic subgradient method also benefits from moderate exponents. At $p=1.25$, however, it is slower than the baseline, and some instances fail to reach the target. This suggests that the tighter prox-linear model can exploit smaller exponents, whereas the subgradient method calls for a moderate choice of $p$.

Taken together, these observations provide empirical support for the paraconvex formulation \cref{eq:exp:obj} of the robust phase retrieval problem under heavy-tailed noise.

\section{Conclusion}

This paper provides the first convergence analysis of stochastic optimization algorithms for minimizing paraconvex objectives. First, we study the stochastic subgradient method for solving unconstrained paraconvex problems. Second, we establish the convergence guarantee of the stochastic proximal subgradient method for composite paraconvex problems with convex regularizers. We then analyze the general stochastic model-based framework, in which the objective is accessed through stochastic models, and derive comprehensive convergence guarantees depending on approximation quality and continuity properties of these models. Numerical experiments on robust phase retrieval in heavy-tailed settings show the potential of using paraconvex objectives in real-world applications.

\begin{acknowledgements}
The first author is supported by Groupe La Poste, the sponsor of the Inria Foundation, in the framework of the FedMalin Inria Challenge, the Inria-Nokia Challenge Learn-Net, and the French National Agency for Research (ANR) under grant ANR-24-CE25-2256 TCDTP. The research of the second author is supported by the Vietnam Ministry of Education and Training under grant number B2026-DQN-01.
\end{acknowledgements}

\bibliographystyle{spmpsci}
\bibliography{references}
\end{document}